\documentclass[a4paper, 10pt, notitlepage]{article}

\usepackage{amsthm, amsmath, amssymb, latexsym}
\usepackage{mathrsfs,cite}

\theoremstyle{plain}
\newtheorem{theorem}{Theorem}
\newtheorem{lemma}{Lemma}

\newtheorem{proposition}{Proposition}

\theoremstyle{definition}
\newtheorem{definition}{Definition}

\theoremstyle{remark}
\newtheorem{remark}{Remark}

\usepackage[a4paper]{geometry}	

\DeclareMathOperator{\dist}{dist}
\DeclareMathOperator{\dom}{dom}
\DeclareMathOperator*{\esssup}{ess\,sup}
\DeclareMathOperator*{\argmin}{arg\,min}

\author{M.V. Dolgopolik and A.V. Fominyh}
\title{Exact Penalty Functions for Optimal Control Problems I: Main Theorem and Free-Endpoint Problems}

\begin{document}

\maketitle

\begin{abstract}
In this two-part study we develop a general approach to the design and analysis of exact penalty
functions for various optimal control problems, including problems with terminal and state constraints, problems
involving differential inclusions, as well as optimal control problems for linear evolution equations. This approach
allows one to simplify an optimal control problem by removing some (or all) constraints of this problem with the use of
an exact penalty function, thus allowing one to reduce optimal control problems to equivalent variational problems,
apply numerical methods for solving, e.g. problems without state constraints, to problems including such constraints,
etc.

In the first part of our study we strengthen some existing results on exact penalty functions for optimisation problems
in infinite dimensional spaces and utilise them to study exact penalty functions for free-endpoint optimal control
problems, which reduce these problems to equivalent variational ones. We also prove several auxiliary results on
integral functionals and Nemytskii operators that are helpful for verifying the assumptions under which the proposed
penalty functions are exact.
\end{abstract}

\section{Introduction}

The idea of using so-called exact penalty functions for solving constrained optimisation problems was suggested
practically simultaneously by Eremin \cite{Eremin} and Zangwill \cite{Zangwill} in the 1960s. Since then, exact penalty
functions have been extensively studied and applied to various optimisation problems by many researchers (see, e.g.
\cite{EvansGouldTolle,HanMangasarian,DiPilloGrippo88,DiPilloGrippo89,ExactBarrierFunc,Dolgopolik_ExPen_I} and the
references therein). 

The main idea behind the exact penalty approach consists in replacing a constrained optimisation problem, say
$$
  \min_{x \in \mathbb{R}^d} \: f(x) \quad \text{subject to} \quad g_i(x) \le 0, \quad i \in \{ 1, 2, \ldots, m \},
$$
by the unconstrained problem of minimising the nonsmooth penalty function:
$$
  \min_{x \in \mathbb{R}^d} \Phi_{\lambda}(x) = f(x) + \lambda \sum_{i = 1}^m \max\{ 0, g_i(x) \}.
$$
Under some natural assumptions this penalised problem is equivalent to the original one in the sense that these
problems have the same optimal value and the same locally/globally optimal solutions, provided the penalty parameter
$\lambda$ is sufficiently large (but finite). Thus, the exact penalty approach allows one to reduce constrained
optimisation problems to equivalent unconstrained ones and apply numerical methods of unconstrained optimisation to
constrained problems. However, most papers on the theory and applications of exact penalty functions deal only with the
finite dimensional case or a local analysis of an exact penalty function. 

In the infinite dimensional case, globally exact penalty functions were probably first studied by Demyanov et al. for 
a problem of finding optimal parameters in a system described by ordinary differential equations
\cite{DemyanovKarelin98}, free-endpoint optimal control problems
\cite{DemyanovKarelin2000_InCollect,DemyanovKarelin2000,Karelin}, the simplest problem of the calculus of variations
\cite{Demyanov2004,Demyanov2005}, and variational problems with nonholonomic inequality constraints
\cite{Demyanov2003,DemyanovGiannessi2003}. The main results of these papers were further extended to isoperimetric
problems of the calculus of variations \cite{DemTam2011}, variational problems involving higher order derivatives
\cite{Tamasyan2013}, parametric moving boundary variational problems \cite{DemTam2014}, control problems involving
differential inclusions \cite{FominyhKarelin2015}, and certain optimal control problems for implicit control systems
with strict inequality constraints \cite{DemyanovTamasyan2005}. Numerical methods for solving optimal control problems
based on the use of exact penalty functions in the infinite dimensional setting were probably first considered by
Outrata \cite{Outrata83} (see also \cite{OutrataSchindler,Outrata88}), and later on were also studied in
\cite{FominyhKarelin2018}. However, in \cite{Outrata83} only the local exactness of a penalty function was considered
under the assumption that an abstract constraint qualification holds true, and it is unclear how to verify this
assumption for any particular problem. In \cite{DemyanovKarelin2000,FominyhKarelin2018}, the global exactness of penalty
functions was stated without proof. The main results on exact penalty functions for various variational problems from
\cite{Demyanov2003,DemyanovGiannessi2003,Demyanov2004,Demyanov2005,DemTam2011,Tamasyan2013,DemTam2014} are based on the
assumptions that the objective function is Lipschitz continuous on a rather complicated and possibly unbounded set, and
a penalty function attains a global minimum in the space of piecewise continuous functions for any sufficiently large
value of the penalty parameter, and it is, once again, unclear how to verify these assumptions in any particular case.
The same remark is true for the main results of the papers 
\cite{DemyanovKarelin98,DemyanovKarelin2000_InCollect,DemyanovKarelin2000,Karelin,DemyanovTamasyan2005,
FominyhKarelin2015} devoted to exact penalty functions for optimal control problems. To the best of authors' knowledge,
the only verifiable sufficient conditions for the global exactness of an exact penalty function in the infinite
dimensional setting were obtained by Gugat and Zuazua in \cite{Zuazua}, where the exact penalisation of the terminal
constraint for optimal control problems involving linear evolution equations was considered.

The main goal of our study is to develop a general theory of exact penalty functions for optimal control problems that
contains verifiable sufficient conditions for the global/complete exactness of penalty functions. To this end, in the
first paper we strengthen some existing results on exact penalty functions for optimisation problems in infinite
dimensional spaces and apply them to free-endpoint problems. We also obtain a number of auxiliary results that are
helpful for verifying the exactness of penalty functions for optimal control problems in particular cases. For instance,
we provide simple sufficient conditions for the Lipschitz continuity of integral functionals, the boundedness of
sublevel sets of penalty functions, the existence of global minimisers, etc. Thus, in this paper we obtain first simple
and verifiable sufficient conditions for the global exactness of penalty functions for nonlinear optimal control
problems, which allow one to reduce free-endpoint optimal control problems to equivalent variational problems. In the
second paper we apply our general results on exact penalty functions to optimal control problems with terminal and
pointwise state constraints, including such problems for linear evolution equations in Hilbert spaces. 

Let us point out that in our study we consider only so-called simple linear penalty functions, i.e. such penalty
functions that depend linearly on the objective function and do not depend on derivatives of the objective function or
constraints. Such exact penalty functions are inherently nonsmooth (see, e.g.~Remark~3 in \cite{Dolgopolik_ExPen_I} and
Sect.~2.3 in \cite{Zuazua}), and one has to utilise a well-developed apparatus of nonsmooth optimisation to minimise
them. In particular, one can apply such popular and efficient modern methods of nonsmooth optimisation as bundle methods
\cite{Makela2002,HaaralaMiettinenMakela,HareSagatizabal,FuduliGaudioso}, 
gradient sampling methods \cite{BurkeLewisOverton,Kiwiel2010,CurtisQue2013}, 
nonsmooth quasi-Newton methods \cite{LewisOverton2013,KeskarWachter}, 
discrete gradient method \cite{BarigovKarasozenSezer} 
(see also \cite{KarmitsaBagriovMakela2012,BagirovKarmitsaMakela_book}), etc. Alternatively, one can utilise smoothing
approximations of nonsmooth penalty functions as in \cite{Pinar,Liu,LiuzziLucidi,Lian,Dolgopolik_ExPen_II}
or the smooth penalty function proposed by Huyer and Neumaier \cite{HuyerNeumaier}. This penalty function was analysed
in detail in \cite{WangMaZhou,Dolgopolik_ExPen_II} and applied to discretised optimal control problems in 
\cite{LiYu2011,JiangLin2012,LinLoxton2014}. In \cite{Dolgopolik_SmoothExPen} it was shown that Huyer and Neumaier's
penalty function is exact if and only if a corresponding standard nonsmooth penalty function is exact. With the use of
this result and the main results of our two-part study one can easily verify the global exactness of Huyer and
Neumaier's penalty function for various optimal control problems without discretisation.

The paper is organised as follows. Some general results on exact penalty functions for optimisation problems in infinite
dimensional spaces are presented in Section~\ref{Section_ExactPenaltyFunctions}. In particular, in this section we
formulate ``the Main Theorem'' (Theorem~\ref{THEOREM_COMPLETEEXACTNESS}), which is the main tool for proving the
global/complete exactness of penalty functions for optimal control problems. We extensively utilise this theorem
throughout both parts of our study. In Section~\ref{Section_FreeEndpoint}, we study an exact penalty function for
free-endpoint optimal control problems, while in Section~\ref{Section_DiffIncl} these results are extended to the case
of free-endpoint variational problems involving differential inclusions. Finally, a proof of the main theorem, as
well as a number of auxiliary results on integral functionals and Nemytskii operators that are helpful for verifying 
the assumptions of the main theorem in the case of optimal control problems, are given in
Appendices~A and B respectively.

\section{Exact Penalty Functions in Metric Spaces}
\label{Section_ExactPenaltyFunctions}

In this section we present some general results on exact penalty functions for optimisation problems in metric spaces
that are utilised throughout the paper. Let $(X, d)$ be a metric space, $M, A \subset X$ be nonempty sets such that
$M \cap A \ne \emptyset$, and $\mathcal{I} \colon X \to \mathbb{R} \cup \{ + \infty \}$ be a given function. Denote by
$\dom \mathcal{I} = \{ x \in X \mid \mathcal{I}(x) < + \infty \}$ the effective domain of $\mathcal{I}$. 

Consider the following optimisation problem:
$$
  \min \: \mathcal{I}(x) \quad \text{subject to} \quad x \in M \cap A.	\eqno{(\mathcal{P})}
$$
Here the sets $M$ and $A$ correspond to two different types of constraints of the optimisation problem. In particular,
it can be equality/inequality constraints or linear/nonlinear constraints or terminal/pointwise constraints.
Denote by $\Omega = M \cap A$ the feasible region of $(\mathcal{P})$. Hereinafter, we suppose that there exists a
globally optimal solution $x^* \in \dom \mathcal{I}$ of the problem $(\mathcal{P})$, i.e. $\mathcal{I}$ attains a global
minimum on $\Omega$, and the optimal value is finite. 

Let a function $\varphi \colon X \to [0, + \infty]$ be such that $\varphi(x) = 0$ iff $x \in M$. For any
$\lambda \ge 0$ introduce the function $\Phi_{\lambda}(x) = \mathcal{I}(x) + \lambda \varphi(x)$. This function is
called \textit{a penalty function} for the problem $(\mathcal{P})$, $\lambda$ is called \textit{a penalty parameter},
and $\varphi$ is called \textit{a penalty term} for the constraint $x \in M$. Note that the function $\Phi_{\lambda}(x)$
is non-decreasing in $\lambda$, $\Phi_{\lambda}(x) \ge \mathcal{I}(x)$ for all $x \in X$, and 
$\Phi_{\lambda}(x) = \mathcal{I}(x)$ for any $x$ satisfying the constraint $x \in M$. Therefore, it is natural to
consider the penalised problem
\begin{equation} \label{PenalizedProblem}
  \min \: \Phi_{\lambda}(x) \quad \text{subject to} \quad x \in A.
\end{equation}
Observe that only the constraint $x \in M$ is penalised, i.e. only this constraint is incorporated into the penalty
function $\Phi_{\lambda}(x)$. This approach allows one to choose which constraints of an optimisation problem are to be
``removed'' via the exact penalty function technique in order to simplify the problem under consideration.

We would like to know when the penalised problem \eqref{PenalizedProblem} is, in some sense, equivalent to 
the original problem $(\mathcal{P})$, i.e. when the penalisation does not distort information about minimisers of
the problem $(\mathcal{P})$.

\begin{definition}
The penalty function $\Phi_{\lambda}$ is called (globally) \textit{exact}, if there exists $\lambda^* \ge 0$ such that
for any $\lambda \ge \lambda^*$ the set of globally optimal solutions of the penalised problem \eqref{PenalizedProblem}
coincides with the set of globally optimal solutions of the problem $(\mathcal{P})$. The greatest lower bound of all
such $\lambda^*$ is denoted by $\lambda^*(\mathcal{I}, \varphi, A)$ and is called \textit{the least exact penalty
parameter} of the penalty function $\Phi_{\lambda}$.
\end{definition}

One can easily verify (see~\cite[Corollary~3.3]{Dolgopolik_ExPen_I}) that the penalty function $\Phi_{\lambda}$ is exact
iff there exists $\lambda \ge 0$ such that $\inf_{x \in A} \Phi_{\lambda}(x) = \inf_{x \in \Omega} \mathcal{I}(x)$, i.e.
iff the optimal values of the problems $(\mathcal{P})$ and \eqref{PenalizedProblem} coincide. Furthermore, the greatest
lower bound of all such $\lambda$ coincides with the least exact penalty parameter.

Thus, if the penalty function $\Phi_{\lambda}$ is globally exact, then for any $\lambda \ge 0$ large enough 
the penalised problem \eqref{PenalizedProblem} is equivalent to the original problem $(\mathcal{P})$ in the sense that
it has the same optimal value and the same globally optimal solutions.

Let us provide simple sufficient conditions for the global exactness of the penalty function $\Phi_{\lambda}$. To this
end, we need to recall the definition of \textit{the rate of steepest descent} of a function defined on a metric space 
\cite{Demyanov2000,Demyanov2010,Uderzo}. Let $g \colon X \to \mathbb{R} \cup \{ + \infty \}$ be a given function, 
$K \subset X$ be a nonempty set, and $x \in K$ be such that $g(x) < + \infty$. The quantity
$$
  g^{\downarrow}_K(x) = \liminf_{y \in K, y \to x} \frac{g(y) - g(x)}{d(y, x)}
$$
is called \textit{the rate of steepest descent} of the function $g$ with respect to the set $K$ at the point $x$ (if
$x$ is an isolated point of $K$, then $g^{\downarrow}_K(x) = + \infty$ by definition). In the case $K = X$ we denote
$g^{\downarrow}(x) = g^{\downarrow}_X(x)$. Let us note that the rate of steepest descent of the function $g$ at $x$ is
closely connected to the so-called strong slope $|\nabla g|(x)$ of $g$ at $x$ \cite{Aze,Kruger}. See
\cite{Aze,Kruger,Dolgopolik_ExPen_II} for some calculus rules for strong slope/rate of steepest descent, and the ways
one can estimate them in various particular cases. Here we only note that if $X$ is a normed space, and $g$ is
Fr\'{e}chet differentiable at a point $x \in X$, then $g^{\downarrow}(x) = - \| g'(x) \|_{X^*}$, where $g'(x)$ is the
Fr\'{e}chet derivative of $g$ at $x$, and $\| \cdot \|_{X^*}$ is the standard norm in the topological dual space $X^*$.
If $g$ is G\^{a}teaux differentiable at $x$, then $g^{\downarrow}(x) \le - \| g'(x) \|_{X^*}$, where $g'(x)$ is the
G\^{a}teaux derivative of $g$ at $x$. Finally, if $g$ is merely directionally differentiable at $x$, then
\begin{equation} \label{RSD_via_DirectDerivative}
  g^{\downarrow}(x) \le \inf_{\| v \| = 1} g'(x, v), \text{ where } 
  g'(x, v) = \lim_{\alpha \to + 0} \frac{g(x + \alpha v) - g(x)}{\alpha}.
\end{equation}
The following theorem, which is a particular case of \cite[Theorem~3.6]{Dolgopolik_ExPen_II}, contains simple sufficient
conditions for the global exactness of the penalty function $\Phi_{\lambda}(x)$. For any $\delta > 0$
define $\Omega_{\delta} = \{ x \in A \mid \varphi(x) < \delta \}$. 

\begin{theorem} \label{Theorem_GlobalExactness}
Let $X$ be a complete metric space, $A$ be closed, $\mathcal{I}$ and $\varphi$ be lower semi-continuous (l.s.c.) on $X$.
Suppose also that there exist a feasible point $x_0 \in X$, $\lambda_0 > 0$ and $\delta > 0$ such that
\begin{enumerate}
\item{the function $\mathcal{I}$ is Lipschitz continuous on an open set containing the set
$C(\delta, \lambda_0) = \{ x \in \Omega_{\delta} \mid \Phi_{\lambda_0}(x) < \mathcal{I}(x_0) \}$;
}

\item{there exists $a > 0$ such that $\varphi^{\downarrow}_A(x) \le - a$ for all 
$x \in C(\delta, \lambda_0) \setminus \Omega$.
}
\end{enumerate}
Then the penalty function $\Phi_{\lambda}$ is globally exact if and only if it is bounded below on $A$ for some 
$\lambda \ge 0$.
\end{theorem}

\begin{remark}
If the assumptions of the theorem above are satisfied, but the penalty function $\Phi_{\lambda}$ is not bounded below,
one can consider the penalty function 
$$
  \Psi_{\lambda}(x) = \begin{cases}
    \mathcal{I}(x) + \lambda \dfrac{\varphi(x)}{\delta - \varphi(x)}, & \text{if } \varphi(x) < \delta, \\
    + \infty, & \text{otherwise.}
  \end{cases}
$$
One can check that under the assumptions of Theorem~\ref{Theorem_GlobalExactness} the penalty function 
$\Psi_{\lambda}$ is exact iff it is bounded below. In particular, $\Psi_{\lambda}$ is exact, provided
the function $\mathcal{I}$ is bounded below on $C(\delta, \lambda_0)$.
\end{remark}

As was noted above, if the penalty function $\Phi_{\lambda}$ is globally exact, then the penalised problem
\eqref{PenalizedProblem} is equivalent to the problem $(\mathcal{P})$ in the sense that it has the same optimal value
and the same \textit{globally} optimal solutions. However, optimisation methods often can find only local
minimisers or even only stationary (critical) points of an optimisation problem. That is why it is desirable to ensure
that local minimisers/stationary points of the penalty function $\Phi_{\lambda}$ coincide with locally optimal
solutions/stationary points of the problem $(\mathcal{P})$. Our aim is to show that this ``complete'' equivalence can be
achieved under assumptions that are very similar to the ones in Theorem~\ref{Theorem_GlobalExactness}. To this end, let
us recall a natural extension of the definition of stationary point to the case of functions defined on metric spaces
(see \cite{Demyanov2000,Demyanov2010}).

Let $g \colon X \to \mathbb{R} \cup \{ + \infty \}$ be a given function, and $K$ be a nonempty set. 
A point $x \in K \cap \dom g$ is called an \textit{inf-stationary} point of the function $g$ on the set $K$, if
$g^{\downarrow}_K(x) \ge 0$. In the case when $X$ is a normed space, $K$ is convex, and $g$ is Fr\'{e}chet
differentiable at $x$ this condition is reduced to the standard necessary optimality condition 
\begin{equation} \label{NessMinCond_OverConvexSet}
  g'(x)[y - x] \ge 0 \quad \forall y \in K.
\end{equation}
Let us also note that if $(\mathcal{P})$ is a mathematical programming problem with equality and inequality constraints,
and $\Phi_{\lambda}$ is the $\ell_1$ penalty function for this problem, then condition 
$\Phi_{\lambda}^{\downarrow}(x) \ge 0$ for some $\lambda > 0$ and a feasible point $x$ is satisfied iff KKT optimality
conditions hold true at $x$.

For any $\lambda \ge 0$ and $c \in \mathbb{R}$ denote $S_{\lambda}(c) = \{ x \in A \mid \Phi_{\lambda}(x) < c \}$. 

\begin{theorem} \label{THEOREM_COMPLETEEXACTNESS}
Let $X$ be a complete metric space, $A$ be closed, $\mathcal{I}$ and $\varphi$ be l.s.c. on $A$, and $\varphi$ be
continuous at every point of the set $\Omega$. Suppose also that there exist 
$c > \mathcal{I}^* = \inf_{x \in \Omega} \mathcal{I}(x)$, $\lambda_0 > 0$, and $\delta > 0$ such that
\begin{enumerate}
\item{$\mathcal{I}$ is Lipschitz continuous on an open set containing the set $S_{\lambda_0}(c) \cap \Omega_{\delta}$;
}

\item{there exists $a > 0$ such that $\varphi^{\downarrow}_A(x) \le - a$ for all 
$x \in S_{\lambda_0}(c) \cap (\Omega_{\delta} \setminus \Omega)$;
\label{NegativeDescentRateAssumpt}}

\item{$\Phi_{\lambda_0}$ is bounded below on $A$.
}
\end{enumerate}
Then there exists $\lambda^* \ge 0$ such that for any $\lambda \ge \lambda^*$ the following statements hold true:
\begin{enumerate}
\item{the optimal values of the problems $(\mathcal{P})$ and \eqref{PenalizedProblem} coincide;
}

\item{globally optimal solutions of the problems $(\mathcal{P})$ and \eqref{PenalizedProblem} coincide;
}

\item{$x^* \in S_{\lambda}(c)$ is a locally optimal solution of the penalised problem \eqref{PenalizedProblem} iff $x^*$
is a locally optimal solution of the problem $(\mathcal{P})$;
}

\item{$x^* \in S_{\lambda}(c)$ is an inf-stationary point of $\Phi_{\lambda}$ on $A$ iff $x^* \in \Omega$, and it is an
inf-stationary point of $\mathcal{I}$ on $\Omega$.
}
\end{enumerate}
\end{theorem}

A proof of Theorem~\ref{THEOREM_COMPLETEEXACTNESS} is given in Appendix~A. If the penalty function
$\Phi_{\lambda}$ satisfies the four statements of this theorem, then it is said to be \textit{completely exact} on the
set $S_{\lambda}(c)$.

\begin{remark}
In the general case, under the assumptions of Theorem~\ref{THEOREM_COMPLETEEXACTNESS} nothing can be
said about locally optimal solutions of the penalised problem \eqref{PenalizedProblem}/inf-stationary points of
$\Phi_{\lambda}$ on $A$ that do not belong to the set $S_{\lambda}(c)$. In order to ensure that the penalty
function $\Phi_{\lambda}$ is completely exact on $A$ (i.e. when $c = +\infty$) one must suppose that the objective
function $\mathcal{I}$ is globally Lipschitz continuous, and there exists $a > 0$ such that 
$\varphi^{\downarrow}_A(x) \le - a$ for all $x \in A \setminus \Omega$ (see \cite[Section~3.3]{Dolgopolik_ExPen_II}).
\end{remark}

\begin{remark}
Let us note that the assumptions of Theorem~\ref{THEOREM_COMPLETEEXACTNESS} cannot be improved (see
\cite[Theorem~3.13]{Dolgopolik_ExPen_II}). On the other hand, the \textit{global} exactness of the penalty function
$\Phi_{\lambda}$ can be proved under weaker assumptions on the penalty term $\varphi$. See~\cite{Dolgopolik_ExPen_I} for
more details.
\end{remark}

In the following section we utilise Theorem~\ref{THEOREM_COMPLETEEXACTNESS} and several auxiliary results on integral
functionals and Nemytskii operators given in Appendix~B to design exact penalty functions for free-endpoint optimal
control problems.

\section{Exact Penalty Functions for Free-Endpoint Optimal Control Problems}
\label{Section_FreeEndpoint}

Consider the following optimal control problem:
\begin{equation} \label{FreeEndPointProblem}
\begin{split}
  {}&\min \: \mathcal{I}(x, u) = \int_0^T \theta(x(t), u(t), t) \, dt + \zeta(x(T)), \\
  {}&\text{subject to } \dot{x}(t) = f(x(t), u(t), t), \quad t \in [0, T], \quad 
  x(0) = x_0, \quad u \in U.
\end{split}
\end{equation}
Here $x(t) \in \mathbb{R}^d$ is the system state at time $t$, $t \to u(t) \in \mathbb{R}^m$ is a control input, 
$f \colon \mathbb{R}^d \times \mathbb{R}^m \times [0, T] \to \mathbb{R}^d$, 
$\theta \colon \mathbb{R}^d \times \mathbb{R}^m \times [0, T] \to \mathbb{R}$, and 
$\zeta \colon \mathbb{R}^d \to \mathbb{R}$ are given functions, while $T > 0$ and $x_0 \in \mathbb{R}^d$ are fixed. We
suppose that $x(\cdot)$ belongs to the space $W_{1, p}^d(0, T)$, which is the Cartesian product of $d$ copies of the
Sobolev space $W^{1, p}(0, T)$, while the control inputs $u(\cdot)$ belong to a closed subset $U$ of the Cartesian
product $L_q^m(0, T)$ of $m$ copies of $L^q(0, T)$. Here $1 < p < + \infty$ and $1 \le q \le + \infty$ (the cases 
$p = 1$ and $p = +\infty$ are excluded to avoid differentiability issues and the use of subdifferentials). For any 
$r \in [1, + \infty]$ denote by $r' \in [1, + \infty]$ the \textit{conjugate exponent} of $r$, i.e. 
$1 / r + 1 / r' = 1$. Also, for any differentiable function $g(x, u, t)$ the gradient of the function 
$x \mapsto g(x, u, t)$ is denoted by $\nabla_x g(x, u, t)$, and a similar notation is used for the gradient of 
the function $u \mapsto g(x, u, t)$.

As usual (see, e.g. \cite{Leoni}), we identify the Sobolev space $W^{1, p}(0, T)$ with the space consisting of all those
absolutely continuous functions $x \colon [0, T] \to \mathbb{R}$ for which $\dot{x} \in L^p(0, T)$. The space 
$L_q^m(0, T)$ with $1 \le q < + \infty$ is equipped with the norm
$$
  \| u \|_q = \left( \int_0^T |u(t)|^q \, dt \right)^{\frac{1}{q}} 
  \quad \forall u \in L_q^m(0, T), 
$$
where $| \cdot |$ is the Euclidean norm, while the space $L_{\infty}^m (0, T)$ is equipped with the norm
$\| u \|_{\infty} = \esssup_{t \in [0, T]}|u(t)|$. The Sobolev space $W_{1, p}^d(0, T)$ is endowed with 
the norm $\| x \|_{1, p} = \| x \|_p + \| \dot{x} \|_p$. Also, below we suppose that the Cartesian product $X \times Y$
of normed spaces $X$ and $Y$ is endowed with the norm $\| (x, y) \| = \| x \|_X + \| y \|_Y$.

\begin{remark} \label{Remark_EquivalentNorms}
For the sake of completeness let us recall two basic facts about norms in Sobolev spaces (see~\cite{Leoni}) that
will be utilised below. From the equality $x(t) = x(0) + \int_0^t \dot{x}(\tau) \, d \tau$ it follows that
$$
  \| x \|_{1, p} \le \big( 1 + \max\{ T, T^{1/p} \} \big) \| x \|_0 \quad \forall x \in W^d_{1, p}(0, T),
$$
where $\| x \|_0 = |x(0)| + \| \dot{x} \|_p$. Hence with the use of the bounded inverse theorem one gets that the norms 
$\| \cdot \|_{1, p}$ and $\| \cdot \|_0$ are equivalent. Moreover, by applying H\"{o}lder's inequality and the equality 
$x(t) = x(0) + \int_0^t \dot{x}(\tau) \, d \tau$ again one obtains that there exists $C > 0$ such that
$\| x \|_{\infty} \le \max\{ 1, T^{1 / p'} \} \| x \|_0 \le C \| x \|_{1, p}$, which implies that any bounded set in
$W^d_{1, p}(0, T)$ is also bounded in $L_{\infty}^d (0, T)$. Let us finally note that from the fact that the operator 
$A \colon L^q(0, T) \to C[0, T]$, $(A x)(t) = \int_0^t x(\tau) \, d \tau$ is compact (which can be easily verified with
the use of the Arzel\`{a}-Ascoli theorem) it follows that for any weakly converging sequence 
$\{ x_n \} \subset W^d_{1, p}(0, T)$ there exists a subsequence $\{ x_{n_k} \}$ strongly converging in $C[0, T]$. This
result is a simple particular case of the Rellich-Kondrachov theorem (see \cite[Theorem~6.2]{Adams}). 
\end{remark}

Our aim is to reduce optimal control problem \eqref{FreeEndPointProblem} to a variational one. To this end we consider
the differential equation $\dot{x}(t) = f(x(t), u(t), t)$ as a constraint that we want to incorporate into a penalty
function. Define $X = W_{1, p}^d(0, T) \times L_q^m(0, T)$ and
$$
  M = \big\{ (x, u) \in X \bigm| F(x, u) = 0 \big\}, \quad
  A = \big\{ (x, u) \in X \bigm| x(0) = x_0, \: u \in U \big\},
$$
where $F(x, u) = \dot{x}(\cdot) - f(x(\cdot), u(\cdot), \cdot)$. Note that the set $A$ is obviously closed.
Problem \eqref{FreeEndPointProblem} can be rewritten as follows:
$$
  \min_{(x, u) \in X} \mathcal{I}(x, u) \quad \text{subject to} \quad (x, u) \in M \cap A.
$$
Formally introduce the penalty term
$$
  \varphi(x, u) = \| F(x, u) \|_p 
  = \left( \int_0^T \big| \dot{x}(t) - f(x(t), u(t), t) \big|^p \, dt \right)^{\frac{1}{p}}.
$$
Clearly, $M = \{ (x, u) \in X \mid \varphi(x, u) = 0 \}$. Therefore one can consider the penalised problem
\begin{equation} \label{FreeEndpointProblem_Penalized}
  \min_{(x, u) \in X} \Phi_{\lambda}(x, u) = \mathcal{I}(x, u) + \lambda \varphi(x, u)
  \quad \text{subject to} \quad (x, u) \in A.
\end{equation}
Observe that this is a variational problem of the form:
\begin{equation} \label{FreEndpointVariationalProblem}
\begin{split}
  {}&\min \: \int_0^T \theta(x(t), u(t), t) \, dt + 
  \lambda \left( \int_0^T \big| \dot{x}(t) - f(x(t), u(t), t) \big|^p \, dt \right)^{\frac{1}{p}} 
  + \zeta(x(T)) \\
  {}&\text{subject to } x(0) = x_0, \quad u \in U,
\end{split}
\end{equation}
With the use of Theorem~\ref{THEOREM_COMPLETEEXACTNESS} we can prove that under some natural assumptions on the
functions $\theta$, $f$ and $\zeta$ this variational problem is equivalent to problem \eqref{FreeEndPointProblem},
provided $\lambda > 0$ large enough. This result allows one to apply methods for solving variational problems to find
optimal solutions of free-endpoint optimal control problems.

\begin{remark}
In most (if not all) optimal control problems appearing in applications, control inputs $u(\cdot)$ are bounded in
$L_{\infty}^m(0, T)$, and the bounds are known in advance, which raises the question of why to consider the case 
$1 \le q < + \infty$. The reason behind this is as follows. Firstly, some authors consider optimal control problems with
only $L^2$-bounded control inputs (see, e.g. \cite{Antipin}), and to apply our results to such problems one must
consider the case $q = 2$. Secondly, the case $q < + \infty$ does not exclude known bounds on control inputs, since one
can define, e.g. $U = \{ u \in L_q^m(0, T) \mid \| u \|_{\infty} \le C \}$ for some $C > 0$, even in the case 
$1 \le q < + \infty$. The reason to suppose $q < + \infty$ in this case is related to the analysis of numerical
methods. Although ``discretise-then-optimise''-type methods are prevalent, there exist some continuous methods for
solving optimal control problems (see, e.g.~\cite{Outrata83,Outrata88,FominyhKarelin2018}), some of which are based on
the minimisation of the penalty function $\Phi_{\lambda}(x, u)$. These methods are usually formulated and analysed in
the case $p = q = 2$, i.e. in the Hilbert space setting, when one can utilise inner products. Therefore, to make
the theory of exact penalty functions consistent with these methods one must consider the case $q = 2$, even in the
presence of known bounds on control inputs. Finally, from the mathematical standpoint it is important to analyse the
general case $1 \le q \le + \infty$ to understand the limitations of the general theory of exact penalty functions and,
in particular, Theorem~\ref{THEOREM_COMPLETEEXACTNESS}.
\end{remark}

Recall that for any $c, \delta > 0$ we define $S_{\lambda}(c) = \{ (x, u) \in A \mid \Phi_{\lambda}(x, u) < c \}$ and
$\Omega_{\delta} = \{ (x, u) \in A \mid \varphi(x, u) < \delta \}$. Note that $\Omega_{\delta}$ consists of all those
$(x, u) \in W^d_{1, p}(0, T) \times U$ that satisfy the perturbed equation
$$
  \dot{x}(t) = f(x(t), u(t), t) + w(t), \quad t \in [0, T], \quad x(0) = x_0
$$
for some $w \in L_p^d(0, T)$ with $\| w \|_p < \delta$. Let $\mathcal{I}^*$ be the optimal value of problem
\eqref{FreeEndPointProblem}. We also need the following definition to conveniently formulate assumptions on the
functions $\theta$ and $f$.

\begin{definition}
Let $g \colon \mathbb{R}^d \times \mathbb{R}^m \times [0, T] \to \mathbb{R}^k$ be a given function. We say that $g$
satisfies \textit{the growth condition} of order $(l, s)$ with $0 \le l < + \infty$ and $1 \le s \le + \infty$, if
for any $R > 0$ there exist $C_R > 0$ and an a.e. nonnegative function $\omega_R \in L^s(0, T)$ such that 
$|g(x, u, t)| \le C_R |u|^l + \omega_R(t)$ for a.e. $t \in [0, T]$ and for all 
$(x, u) \in \mathbb{R}^d \times \mathbb{R}^m$ with $|x| \le R$.
\end{definition}

With the use of several auxiliary results on integral functionals and Nemytskii operators from Appendix~B we can prove
the following theorem, which provides conditions under which free-endpoint optimal control problem
\eqref{FreeEndPointProblem} and variational problem \eqref{FreEndpointVariationalProblem} are equivalent.

\begin{theorem} \label{Thrm_FreeEndPointProblem}
Let the following assumptions be valid:
\begin{enumerate}
\item{$\zeta$ is locally Lipschitz continuous, $\theta$ is continuous and differentiable in $x$ and $u$, $f$ is
continuous and differentiable in $x$, and the functions $\nabla_x \theta$, $\nabla_u \theta$ and $\nabla_x f$ are
continuous;
\label{Assumpt_FEPP_Differentiability}}

\item{either $q = + \infty$ or the functions $\theta$ and $\nabla_x \theta$ satisfy the growth condition of order 
$(q, 1)$, the function $\nabla_u \theta$ satisfies the growth condition of order $(q - 1, q')$, and the functions $f$
and $\nabla_x f$ satisfy the growth condition of order $(q / p, p)$;
\label{Asumpt_FEPP_GrowthConditions}}

\item{there exists a globally optimal solution of problem \eqref{FreeEndPointProblem};}

\item{there exist $\lambda_0 > 0$, $c > \mathcal{I}^*$, and $\delta > 0$ such that the set 
$S_{\lambda_0}(c) \cap \Omega_{\delta}$ is bounded in $W^d_{1, p}(0, T) \times L_q^m(0, T)$, and the function
$\Phi_{\lambda_0}(x, u)$ is bounded below on $A$.
\label{Asumpt_FEPP_SublevelBounded}}
\end{enumerate}
Then there exists $\lambda^* \ge 0$ such that for any $\lambda \ge \lambda^*$ the penalty function $\Phi_{\lambda}$
for problem \eqref{FreeEndPointProblem} is completely exact on $S_{\lambda}(c)$.
\end{theorem}

\begin{proof}
Our aim is to apply Theorem~\ref{THEOREM_COMPLETEEXACTNESS}. By Proposition~\ref{Prp_Functional_Correct} below the
growth condition on $\theta$ ensures that the functional $\mathcal{I}(x, u)$ is correctly defined and finite 
for any $(x, u) \in X$. In turn, the growth conditions on $\nabla_x \theta$ and $\nabla_u \theta$ guarantee that the
mapping $(x, u) \mapsto \int_0^T \theta(x(t), u(t), t) \, dt$ is Lipschitz continuous on any bounded subset of
$L_{\infty}^d(0, T) \times L_q^m(0, T)$ by Proposition~\ref{Prp_LipschitzConditions}. From
Remark~\ref{Remark_EquivalentNorms} it follows that any bounded subset of $X$ is bounded in 
$L_{\infty}^d(0, T) \times L_q^m(0, T)$. Therefore by applying the fact that $\zeta$ is locally Lipschitz continuous one
obtains that the functional $\mathcal{I}(x, u)$ is correctly defined and Lipschitz continuous on any bounded subset of
$X$ (in particular, on any bounded open set containing $S_{\lambda_0}(c) \cap \Omega_{\delta}$; recall that
$S_{\lambda_0}(c) \cap \Omega_{\delta}$ is bounded by our assumption). Finally, by applying the growth condition on the
function $f$ and Proposition~\ref{Prp_DiffEqConstr_Correct} one gets that the penalty term $\varphi(x, u)$ is correctly
defined and continuous on $X$. Let us check that for any bounded set $K \subset A$ there exists $a > 0$ such that
$\varphi^{\downarrow}_A(x, u) \le - a$ for any $(x, u) \in K \setminus \Omega$. Then by applying
Theorem~\ref{THEOREM_COMPLETEEXACTNESS} one gets the desired result.

Define $A_s = \{ x \in W^d_{1, p}(0, T) \mid x(0) = x_0 \}$, and for any $u \in U$ introduce the function
$\varphi_u(x) = \varphi(x, u)$. Observe that $\varphi^{\downarrow}_A(x, u) \le (\varphi_u)^{\downarrow}_{A_s}(x)$ for
any $(x, u) \in A$ due to the fact that $A = A_s \times U$. Therefore, it is sufficient to check that for any
bounded set $K \subset A$ there exists $a > 0$ such that $(\varphi_u)^{\downarrow}_{A_s}(x) \le - a$ for all 
$(x, u) \in K \setminus \Omega$, i.e. for all $(x, u) \in K$ such that $\varphi_u(x) > 0$. To simplify the computation
of $(\varphi_u)^{\downarrow}_{A_s}(x)$ we apply a change of variables called ``transition into the space of
derivatives'' that was widely utilised by Demyanov in his works on exact penalty functions
(see, e.g. \cite{Demyanov2004,Demyanov2005,Demyanov2003}).

For any $z \in L_p^d(0, T)$ define $(J z)(t) = x_0 + \int_0^t z(\tau) d \tau$ for all $t \in [0, T]$ and 
$\gamma_u(z) = \varphi_u(J z)$. From the Lebesgue differentiation theorem it follows that the operator $J$ is a
one-to-one correspondence between $L_p^d (0, T)$ and $A_s$ (see~\cite{Leoni}). Furthermore, by H\"{o}lder's inequality
one has $\| Jz - Jw \|_{1, p} \le (1 + T) \| z - w \|_p$ for any $z, w \in L_p^d(0, T)$. Consequently, if for
some $x \in A_s$ and $u \in U$ one has $(\gamma_u)^{\downarrow}(\dot{x}) < 0$, then
\begin{align*}
  0 > (\gamma_u)^{\downarrow}(\dot{x}) 
  = \liminf_{z \to \dot{x}} \frac{\gamma_u(z) - \gamma_u(\dot{x})}{\| z - \dot{x} \|_p}
  &= \liminf_{y \in A_s, y \to x} \frac{\gamma_u(\dot{y}) - \gamma_u(\dot{x})}{\| \dot{y} - \dot{x} \|_p} \\
  &\ge (1 + T) \liminf_{y \in A_s, y \to x} \frac{\varphi_u(y) - \varphi_u(x)}{\| y - x \|_{1, p}}
  = (1 + T) (\varphi_u)^{\downarrow}_{A_s}(x).
\end{align*}
Therefore, it is sufficient to check that for any bounded sets $Z \subset L_p^d(0, T)$ and $V \subseteq U$ there
exists $a > 0$ such that $(\gamma_u)^{\downarrow}(z) \le - a$ for all $(z, u) \in Z \times V$ such that 
$\gamma_u(z) > 0$ (note that the transition into the space of derivatives allowed us to ``remove'' the constraint 
$x \in A_s$).

Fix any bounded sets $Z \subset L_p^d(0, T)$ and $V \subseteq U$. Introduce the function
$H \colon \mathbb{R}^d \to \mathbb{R}^d$,
$$
  H(x) = \begin{cases}
    |x|^{p - 2} x, & \text{if } x \ne 0, \\
    0, & \text{if } x = 0,
  \end{cases}
$$
and for any $(z, u) \in Z \times V$ define $G(z, u)(\cdot) = H(F(Jz, u)(\cdot))$, where, as above, 
$F(x, u) = \dot{x}(\cdot) - f(x(\cdot), u(\cdot), \cdot)$. It is easy to verify that $H$ is a continuous function
(recall that $p > 1$), which implies that that the function $G(z, u)(\cdot)$ is measurable. Furthermore, for
any $x \in L_p^d(0, T)$ one has $H(x(\cdot)) \in L_{p'}^d (0, T)$, which by the growth condition on $f$ and
Proposition~\ref{Prp_DiffEqConstr_Correct} implies that $G(z, u) \in L_{p'}^d (0, T)$ for any $(z, u) \in Z \times V$.
Note also that
$$
  \frac{1}{\gamma_u(z)^{p - 1}} \int_0^T \big\langle G(z, u)(t), F(Jz, u)(t) \big\rangle \, dt = \| F(Jz, u) \|_p =
\gamma_u(z),
$$
provided $\gamma_u(z) > 0$. Here $\langle \cdot, \cdot \rangle$ is the inner product in $\mathbb{R}^d$.

Taking into account the growth condition on $\nabla_x f$ and
Proposition~\ref{Prp_DiffEqConstr_Diferentiable} (see Appendix~B) one gets that the mapping $F_u(x) = F(x, u)$ is
G\^{a}teaux differentiable for any $(x, u) \in X$, and its G\^{a}teaux derivative has the form
$$
  F'_u(x)[h] = \dot{h}(\cdot) - \nabla_x f(x(\cdot), u(\cdot), \cdot) h(\cdot) \quad \forall h \in W_{1, p}^d(0, T).
$$
Hence by applying the fact that the norm $\| \cdot \|_p$ is Fr\'{e}chet differentiable (this fact follows, e.g. from
\cite[Theorem~2.2.1]{Diestel}) and the chain rule one obtains that for any $(z, u) \in Z \times V$ such that 
$\gamma_u(z) > 0$ the function $\gamma_u$ is G\^{a}teaux differentiable at $z$, and
$$
  \gamma_u'(z)[h] = 
  \frac{1}{\gamma_u(z)^{p - 1}} \int_0^T \Big\langle G(z, u)(t), 
  h(t) - \nabla_x f (x(t), u(t), t) \int_0^t h(\tau) \, d \tau \Big\rangle \, dt
  \quad \forall h \in L_p^d(0, T)
$$
(here $x = J z$). Integrating by parts one obtains that
$$
  \gamma_u'(z)[h] = \int_0^T 
  \left\langle w(x, u)(t) - \int_t^T \nabla_x f (x(\tau), u(\tau), \tau)^T w(x, u)(\tau) \, d \tau,
  h(t) \right\rangle \, dt,
$$
where $w(x, u) = \gamma_u(z)^{1 - p} G(x, u)$. Consequently, taking into account the fact that $\| w(x, u) \|_{p'} = 1$
one gets that
$$
  (\gamma_u)^{\downarrow}(z) \le -\| \gamma_u'(z) \| = -\| (I - \mathcal{K}_y) w(x, u) \|_{p'} 
  \le - \frac{1}{\| (I - \mathcal{K}_y)^{-1} \|},
$$
where $I$ is the identity operator and
$$
  \big( \mathcal{K}_y h \big)(t) = \int_t^T y(s) h(s) \, ds \quad \forall h \in L_{p'}^d(0, T), 
  \quad y(s) = \nabla_x f(x(s), u(s), s)^T.
$$
From the facts that $\nabla_x f$ satisfies the growth condition of order $(q/p, p)$, and the sets $Z$ and $V$
are bounded it follows that the set $\{ \nabla_x f(Jz(\cdot), u(\cdot), \cdot) \mid (z, u) \in Z \times V \}$ of
kernels of the integral operators $\mathcal{K}_y$ is bounded in $L_p^{d \times d}(0, T)$. Hence by applying
Lemma~\ref{Lemma_BoundedResolvent} on the resolvent of a Volterra-type integral equation one obtains that there exists
$a > 0$ such that $\| (I - \mathcal{K}_y)^{-1} \| < 1 / a$, i.e. $(\gamma_u)^{\downarrow}(z) \le - a$, for any 
$(z, u) \in Z \times V$, and the proof is complete.
\end{proof}

\begin{remark}
Let $1 < q < + \infty$, the set $U$ be convex (or, more generally, weakly closed), and the following
assumptions be valid:
\begin{enumerate}
\item[(i)]{$f(x, u, t) = f_0(x, t) + g(x, t) u$, where the functions $f_0$ and $g$ are continuous;}

\item[(ii)]{$\theta(x, u, t)$ is convex in $u$ for all $x \in \mathbb{R}^d$ and $t \in [0, T]$.}
\end{enumerate}
Then under assumptions \ref{Assumpt_FEPP_Differentiability}, \ref{Asumpt_FEPP_GrowthConditions}, and
\ref{Asumpt_FEPP_SublevelBounded} of Theorem~\ref{Thrm_FreeEndPointProblem} a globally optimal solution of problem
\eqref{FreeEndPointProblem} exists iff there exists a feasible point of this problem, i.e. iff there exists 
$u \in U$ such that there exists an absolutely continuous solution of the differential equation 
$\dot{x} = f(x, u, t)$ with $x(0) = x_0$ defined on $[0, T]$.

Indeed, if a feasible point exists, then the sublevel set 
$\{ (x, u) \in \Omega \mid \mathcal{I}(x, u) < c \} \subset S_{\lambda_0}(c) \cap \Omega_{\delta}$ is nonempty and
bounded. Let $\{ (x_n, u_n) \} \subset \Omega$ be a sequence such that $\mathcal{I}(x_n, u_n) \to \mathcal{I}^*$ as 
$n \to \infty$. Since $c > \mathcal{I}^*$, the sequence $\{ (x_n, u_n) \}$ is bounded, which due to the reflexivity of
$L_q^m(0, T)$ and $W^d_{1, p}(0, T)$ for $1 < q, p < + \infty$ implies that one can extract 
a subsequence $\{ (x_{n_k}, u_{n_k}) \}$ weakly converging to some $(x^*, u^*)$. Note that $u^* \in U$, since
$U$ is weakly closed. Furthermore, by Remark~\ref{Remark_EquivalentNorms} one can suppose that $x_{n_k}$
converges to $x^*$ uniformly on $[0, T]$. Hence by applying assumption~(i) one can easily check that $(x^*, u^*)$ is a
feasible point of problem \eqref{FreeEndPointProblem}, while assumption~(ii) ensures 
that $\mathcal{I}(x^*, u^*) \le \liminf_{k \to \infty} \mathcal{I}(x_{n_k}, u_{n_k}) = \mathcal{I}^*$
(see~\cite[Section~7.3.2]{FonsecaLeoni} and \cite{Ioffe77}), which implies that $(x^*, u^*)$ is a globally optimal
solution of problem \eqref{FreeEndPointProblem}.

Note finally that the existence of a feasible point of problem \eqref{FreeEndPointProblem} can be proved with the use
of various standard results from the theory of differential equations. For example, it exists, if there
exist $u \in U$, $C > 0$, and a.e. nonnegative function $\omega \in L^1(0, T)$ such that 
$|f(x, u(t), t)| \le C|x| + \omega(t)$ for a.e. $t \in (0, T)$ and for all $x \in \mathbb{R}^d$ 
(cf. Proposition~\ref{Proposition_LevelBoundness} and Remark~\ref{Remark_BihariLaSalle} below). 
\end{remark}

Let us also point out several simple ways to verify the boundedness of the set $S_{\lambda}(c) \cap \Omega_{\delta}$
and the boundedness below of $\Phi_{\lambda}(x, u)$ on the set $A$. One can utilise a combination of these ways and a
structure of the problem in order to verify the boundedness conditions for particular optimal control problems.

\begin{proposition} \label{Proposition_LevelBoundness}
Let $\theta$ and $f$ be continuous, and one of the following assumptions be valid:
\begin{enumerate}
\item{the set $U$ is bounded in $L_{\infty}^m(0, T)$, and for any $R > 0$ there exist $C_R > 0$ and an a.e.
nonnegative function $\omega_R \in L^1(0, T)$ such that
\begin{equation} \label{LinearGrowthRHS_BoundedControl}
  |f(x, u, t)| \le C_R |x| + \omega_R(t), \quad
  \theta(x, u, t) \ge - C_R |x| - \omega_R(t)
\end{equation}
for all $(x, u) \in \mathbb{R}^d \times \mathbb{R}^m$ with $|u| \le R$ and for a.e. $t \in (0, T)$, and there exist
$K_1, K_2 \ge 0$ such that $\zeta(x) \ge - K_1 |x| - K_2$ for all $x \in \mathbb{R}^d$;
\label{Assumpt_BoundedControl}}

\item{$1 \le q < + \infty$, there exist $C_f > 0$ and a.e. nonnegative function $\omega_f \in L^p(0, T)$ such that
\begin{equation} \label{qpGrowth_in_Control}
  |f(x, u, t)| \le C_f \big( |x| + |u|^{\frac{q}{p}} \big) + \omega_f(t)
\end{equation}
for all $(x, u) \in \mathbb{R}^d \times \mathbb{R}^m$ and a.e. $t \in (0, T)$, and one of the two following assumptions
is valid:
\begin{enumerate}
\item{$U$ is bounded in $L_q^m(0, T)$, and there exist $C, K_1, K_2 > 0$, and an a.e. nonnegative function 
$\omega \in L^1(0, T)$ such that for all $(x, u)$ and a.e. $t \in (0, T)$ 
one has $\theta(x, u, t) \ge - C( |x| + |u|^q ) - \omega(t)$ and $\zeta(x) \ge - K_1 |x| - K_2$;
}

\item{$\zeta$ is bounded below, and there exist $C > 0$ and $\omega \in L^1(0, T)$ such that 
$\theta(x, u, t) \ge C |u|^q + \omega(t)$ for all $(x, u) \in \mathbb{R}^d \times \mathbb{R}^m$ and a.e. $t \in (0, T)$.
}
\end{enumerate}
\label{Assumpt_L_q_BoundedControl}}

\item{$1 \le q < + \infty$, $\zeta$ is bounded below, there exist $s \ge r \ge 1$, $C > 0$, and an a.e. nonnegative
function $\omega \in L^p(0, T)$ such that
\begin{equation} \label{Growth_in_State_and_Control}
  |f(x, u, t)| \le C\big( |x|^r + |u|^{\frac{q}{p}} \big) + \omega(t), \quad
  \theta(x, u, t) \ge C\big( |x|^s + |u|^q \big) - \omega(t)
\end{equation}
for all $(x, u) \in \mathbb{R}^d \times \mathbb{R}^m$ and a.e. $t \in (0, T)$.
\label{Assumpt_UnboundedControl}}
\end{enumerate}
Then there exists $\lambda_0 \ge 0$ such that for any $c \in \mathbb{R}$, $\delta > 0$, and $\lambda \ge \lambda_0$ the
set $S_{\lambda}(c) \cap \Omega_{\delta}$ is bounded, and the function $\Phi_{\lambda}$ is bounded below on $A$.
Furthermore, if either assumption~\ref{Assumpt_BoundedControl} or assumption~\ref{Assumpt_L_q_BoundedControl} is
satisfied, then there exists a feasible point of problem \eqref{FreeEndPointProblem}.
\end{proposition}

\begin{proof}
\textbf{Part~\ref{Assumpt_BoundedControl}}.~Fix $\delta > 0$. By definition for any $(x, u) \in \Omega_{\delta}$ one
has $\| F(x, u) \| < \delta$. Hence there exists $w \in L^d_p(0, T)$ with $\| w \|_p < \delta$ such that
\begin{equation} \label{PerturbedEquation}
  \dot{x}(t) = f(x(t), u(t), t) + w(t) \quad \text{for a.e. } t \in [0, T]
\end{equation}
or, equivalently,
\begin{equation} \label{PerturbedEq_IntegralForm}
  x(t) = x_0 + \int_0^t f(x(\tau), u(\tau), \tau) d \tau + \int_0^t w(\tau) \, d \tau \quad \forall t \in [0, T].
\end{equation}
Since $U$ is bounded in $L_{\infty}^m(0, T)$, there exists $R > 0$ such that for all $u \in U$ one
has $|u(t)| \le R$ for a.e. $t \in (0, T)$. Consequently, there exists $C_R > 0$ and an a.e. nonnegative function
$\omega_R \in L^1(0, T)$ such that for any $(x, u) \in \Omega_{\delta}$ one has
\begin{equation} \label{IntegralInequalOnPerturbedSol}
  |x(t)| \le |x_0| + \int_0^t \big( C_R |x(\tau)| + \omega_R(\tau) + |w(\tau)| \big) \, d \tau
  \quad \forall t \in [0, T]
\end{equation}
for some $w \in L^d_p(0, T)$ with $\| w \|_p < \delta$. By applying the Gr\"{o}nwall-Bellman inequality one obtains that
\begin{equation} \label{GronwallInequalPerturbedSol}
  |x(t)| \le \alpha(t) + C_R \int_0^t \alpha(\tau) e^{C_R(t - \tau)} \, d \tau \quad \forall t \in [0, T],
\end{equation}
where $\alpha(t) = |x_0| + \int_0^t (\omega_R(\tau) + |w(\tau)|) \, d \tau$. With the use of H\"{o}lder's inequality one
gets that 
\begin{equation} \label{GronwallInequal_NonintegralTerm}
  \| \alpha \|_{\infty} \le |x_0| + \| \omega_R \|_1 + T^{\frac{1}{p'}} \delta.
\end{equation}
Therefore, the set $\Omega_{\delta}$ is bounded in $L_{\infty}^d (0, T) \times L_{\infty}^m (0, T)$, which implies that
there exists $C > 0$ such that $|f(x(t), u(t), t)| \le C$ for a.e. $t \in [0, T]$ and for all 
$(x, u) \in \Omega_{\delta}$. Hence and from \eqref{PerturbedEquation} it follows that for all 
$(x, u) \in \Omega_{\delta}$ one has
$$
  |\dot{x}(t)|^p \le 2^p (C^p + |w(t)|^p) \quad \text{for a.e. } t \in [0, T].
$$
Integrating this inequality from $0$ to $T$ and taking into account the fact that $\| w \|_p < \delta$ one obtains that
$\Omega_{\delta}$ is bounded in $X$. Thus, $S_{\lambda}(c) \cap \Omega_{\delta}$ is bounded in $X$ for any 
$\lambda \ge 0$, $c \in \mathbb{R}$ and $\delta > 0$. 

Fix $(x, u) \in A$, and let $\delta = \varphi(x, u)$. From \eqref{GronwallInequalPerturbedSol} and
\eqref{GronwallInequal_NonintegralTerm} it follows that there exist $C_1, C_2 > 0$ depending only on $C_R$, $\omega_R$
and $T$ such that $\| x \|_{\infty} \le C_1 + C_2 \delta$. By applying the second inequality in
\eqref{LinearGrowthRHS_BoundedControl} one obtains that
\begin{align*}
  \Phi_{\lambda}(x, u) = \mathcal{I}(x, u) + \lambda \varphi(x, u) 
  &\ge - T C_R (C_1 + C_2 \delta) - \| \omega_R \|_1
  - K_1 (C_1 + C_2 \delta) - K_2 + \lambda \delta \\
  &\ge - C_1 (T C_R + K_1) - \| \omega_R \|_1 - K_2
\end{align*}
for any $\lambda \ge C_2 (T C_R + K_1)$. Consequently, the penalty function $\Phi_{\lambda}$ is bounded below on $A$ 
for any $\lambda \ge C_2 (T C_R + K_1)$.

Let us now prove the existence of a feasible point. Fix $u \in U$. From the fact that $f$ is continuous it
follows that a solution $x(\cdot)$ of \eqref{PerturbedEquation} with $w(\cdot) \equiv 0$ and $x(0) = x_0$ is defined at
least on some subinterval $[0, T_0)$ of $[0, T]$. By applying \eqref{IntegralInequalOnPerturbedSol} and 
the Gr\"{o}nwall-Bellman inequality one can easily check that  $T_0 = T$, and $x(\cdot)$ is bounded on $[0, T)$.
Furthermore, from the continuity of $f$ it obviously follows that $x \in W^d_{1, \infty}(0, T)$, which implies that 
$(x, u)$ is a feasible point of problem \eqref{FreeEndPointProblem}.

\textbf{Part~\ref{Assumpt_L_q_BoundedControl}}.~Fix $c \in \mathbb{R}$ and $\delta > 0$. By applying either
the boundedness of the set $U$ in $L_q^m(0, T)$ or the boundedness below of $\zeta$ and the
inequalities $\theta(x, u, t) \ge C |u|^q + \omega(t)$ and $\Phi_{\lambda}(\cdot) \ge \mathcal{I}(\cdot)$ one obtains
that there exists $K > 0$ such that $\| u \|_q \le K$ for all $(x, u) \in S_{\lambda}(c)$ and any $\lambda \ge 0$.

As was pointed out above, for any $(x, u) \in \Omega_{\delta}$ there exists $w \in L^d_p(0, T)$ with
$\| w \|_p < \delta$ such that \eqref{PerturbedEq_IntegralForm} holds true. By applying \eqref{qpGrowth_in_Control} one
gets that
\begin{equation} \label{IntegralInequalOnPerturbedSol_Assumpt2}
  |x(t)| \le \alpha(t) + \int_0^t C_f |x(\tau)| \, d \tau, \quad
  \alpha(t) = |x_0| + \int_0^t \big( \omega_f(\tau) + C_f |u(\tau)|^{\frac{q}{p}} + |w(\tau)| \big) \, d \tau
\end{equation}
for any $t \in [0, T]$. Hence with the use of the Gr\"{o}nwall-Bellman and H\"{o}lder's inequalities one obtains that
$$
  |x(t)| \le \alpha(t) + C_f \int_0^t \alpha(\tau) e^{C_f(t - \tau)} \, d \tau \quad \forall t \in [0, T],
  \qquad
  \| \alpha \|_{\infty} \le |x_0| + T^{\frac{1}{p'}} \big( \| \omega_f \|_p + C_f K^{\frac{q}{p}} + \delta \big).
$$
for any $(x, u) \in S_{\lambda}(c) \cap \Omega_{\delta}$, which implies that the set 
$S_{\lambda}(c) \cap \Omega_{\delta}$ is bounded in $L_{\infty}^d(0, T) \times L_q^m(0, T)$ for any $\lambda \ge 0$.
Hence by applying \eqref{qpGrowth_in_Control}, \eqref{PerturbedEquation}, and H\"{o}lder's inequality one can easily
check that this set is bounded in $X$ for any $\lambda \ge 0$.

If $\zeta$ is bounded below and $\theta(x, u, t) \ge C |u|^q + \omega(t)$, then the boundedness below of the penalty
function $\Phi_{\lambda}$ follows from the inequality $\Phi_{\lambda}(\cdot) \ge \mathcal{I}(\cdot)$. On the other
hand, if $U$ is bounded, and the inequalities $\theta(x, u, t) \ge - C( |x| + |u|^q ) - \omega(t)$ and
$\zeta(x) \ge - K_1 |x| - K_2$ are satisfied, then the boundedness below of the penalty function $\Phi_{\lambda}$ can
be proved in the same way as in part~\ref{Assumpt_BoundedControl} of the proposition. 

Finally, let us prove the existence of a feasible point. Fix any $u \in U$. From the continuity of $f$ and
\eqref{qpGrowth_in_Control} it follows that the function $(t, x) \mapsto f(x, u(t), t)$ satisfies the Carath\'{e}odory
condition, and $|f(x, u(t), t)| \le C_f |x| + \eta(t)$ for all $x \in \mathbb{R}^d$ and a.e. $t \in (0, T)$, where
$\eta(\cdot) = C_f |u(\cdot)|^{q/p} + \omega_f(\cdot) \in L^p(0, T)$. Consequently, there exists an absolutely
continuous solution $x(\cdot)$ of \eqref{PerturbedEquation} with $w(\cdot) \equiv 0$ and $x(0) = x_0$ defined at least
on some subinterval $[0, T_0)$ of $[0, T]$. By applying \eqref{IntegralInequalOnPerturbedSol_Assumpt2} and 
the Gr\"{o}nwall-Bellman inequality one obtains that $T_0 = T$, and $x(\cdot)$ is bounded on $[0, T)$. Hence and from
\eqref{qpGrowth_in_Control} and \eqref{PerturbedEquation} with $w(\cdot) \equiv 0$ it follows that 
$x \in W^d_{1, p}(0, T)$, which implies that $(x, u)$ is a feasible point of problem \eqref{FreeEndPointProblem}.

\textbf{Part~\ref{Assumpt_UnboundedControl}}.~Fix $c \in \mathbb{R}$ and $\delta > 0$. From the second
inequality in \eqref{Growth_in_State_and_Control} and the fact that $\zeta$ is bounded below it obviously follows that
for any $\lambda \ge 0$ the penalty function $\Phi_{\lambda}$ is bounded below on $A$, and the set $S_{\lambda}(c)$ is
bounded in $L_s^d(0, T) \times L_q^m(0, T)$.

By applying \eqref{PerturbedEq_IntegralForm} and the first inequality in \eqref{Growth_in_State_and_Control} one obtains
that for any $(x, u) \in \Omega_{\delta}$ there exists $w \in L^d_p(0, T)$ with $\| w \|_p < \delta$ such that
$$
  |x(t)| \le 
  \int_0^t \big( C\big( |x(\tau)|^r + |u(\tau)|^{\frac{q}{p}} \big) + \omega(\tau) + |w (\tau)| \big) \, d \tau
$$
for any $t \in (0, T)$. Hence with the use of H\"{o}lder's inequality and the fact that $s \ge r$ one gets that there
exists $K > 0$ such that $\| x \|_{\infty} \le K$ for any $(x, u) \in S_{\lambda}(c) \cap \Omega_{\delta}$.
Therefore, by applying \eqref{PerturbedEquation} and the first inequality in \eqref{Growth_in_State_and_Control} one
obtains that
$$
  |\dot{x}(t)| \le C K^r + C |u(t)|^{\frac{q}{p}} + \omega(t) + |w(t)|
$$
for a.e. $t \in (0, T)$, where $(x, u) \in S_{\lambda}(c) \cap \Omega_{\delta}$ and $\| w \|_p < \delta$. Consequently,
taking into account the facts that $\omega \in L^p(0, T)$ and the set $(x, u) \in S_{\lambda}(c) \cap \Omega_{\delta}$
is bounded in $L_{\infty}^d(0, T) \times L_q^m(0, T)$, one obtains that there exists $R > 0$ such that 
$\| \dot{x} \|_p \le R$ for any $S_{\lambda}(c) \cap \Omega_{\delta}$, i.e. this set is bounded in $X$.
\end{proof}

\begin{remark} \label{Remark_BihariLaSalle}
Let us note that the assumptions of the first two parts of the proposition above can be relaxed. For example, let 
$R = \sup\{ \| u \|_{\infty} \mid u \in U \}$, and suppose that instead of the first inequality in 
\eqref{LinearGrowthRHS_BoundedControl} the inequality $|f(x, u, t)| \le \eta_R(|x|) + \omega_R(t)$ holds true for all 
$(x, u) \in \mathbb{R}^d \times \mathbb{R}^m$ with $|u| \le R$ and for a.e. $t \in (0, T)$, where 
$\omega_R \in L^1(0, T)$, and $\eta_R \colon [0, + \infty) \to [0, + \infty)$ is a continuous non-decreasing function
such that $\eta_R(s) > 0$ for any $s > 0$. Then arguing in the same way as in the proof of
Proposition~\ref{Proposition_LevelBoundness}, but utilising the Bihari-LaSalle inequality instead of 
the Gr\"{o}nwall-Bellman inequality one can easily verify that the set $\Omega_{\delta}$ is bounded, provided $T > 0$
satisfies the assumptions of the Bihari-LaSalle inequality. However, to ensure the boundedness below of the penalty
function $\Phi_{\lambda}$ in this case one must suppose that both functions $\theta$ and $\zeta$ are bounded below.
\end{remark}

\begin{remark}
Proposition~\ref{Proposition_LevelBoundness} demonstrates how one can prove the boundedness of the set
$S_{\lambda}(c) \cap \Omega_{\delta}$ and the boundedness below of the function $\Phi_{\lambda}$ using various
information about the functions $f$, $\theta$, $\zeta$, and the set of admissible control inputs $U$. In the first part
of this proposition we suppose that the set $U$ is bounded in $L_{\infty}^m(0, T)$, which, in essence, allows us not to
impose any assumptions on the behaviour of the functions $u \mapsto f(x, u, t)$ and $u \mapsto \theta(x, u, t)$. If $U$
is bounded only in $L_q^m(0, T)$, then appropriate growth conditions on these functions must be imposed to prove the
required result (see part two of Proposition~\ref{Proposition_LevelBoundness}). Finally, if no information about 
the set $U$ is available, one must impose some assumptions that ensure the coercivity of the functional 
$\mathcal{I}(x, u)$, as it is done in the last part of Proposition~\ref{Proposition_LevelBoundness}. 
\end{remark}

\section{Exact Penalty Functions for Free-Endpoint Differential Inclusions}
\label{Section_DiffIncl}

Let us extend the main results of the previous section to the case of variational problems involving differential
inclusions. Consider the following problem:
\begin{equation} \label{DiffInclProblem}
\begin{split}
  {}&\min \: \mathcal{I}(x) = \int_0^T \theta(x(t), t) \, dt + \zeta(x(T)), \\
  {}&\text{subject to } \dot{x}(t) \in F(x(t), t), \quad t \in [0, T], \quad x(0) = x_0.
\end{split}
\end{equation}
Here $\theta \colon \mathbb{R}^d \times [0, T] \to \mathbb{R}$ and $\zeta \colon \mathbb{R}^d \to \mathbb{R}$ are given
functions, $F \colon \mathbb{R}^d \times [0, T] \rightrightarrows \mathbb{R}^d$ is a set-valued mapping with nonempty
compact convex values, $T > 0$ and $x_0 \in \mathbb{R}^d$ are fixed, and $x \in W^d_{1,p}(0, T)$.

\begin{remark}
Although problem \eqref{DiffInclProblem} does not include control inputs, it encompasses many optimal control problems,
including some optimal feedback control problems, which cannot be tackled with the use of the approach presented in 
the previous section. For example, the system
$$
  \dot{x}(t) = f(x(t), u(t), t), \quad u(t) \in U(x(t), t),
$$
where $U \colon \mathbb{R}^d \times [0, T] \rightrightarrows \mathbb{R}^m$ is a multifunction, can be rewritten as 
the differential inclusion $\dot{x}(t) \in F(x(t), t)$ with $F(x, t) = f(x, U(x, t), t)$, which allows one to reduce
optimal control problems involving such systems to problem \eqref{DiffInclProblem}. Thus, the results of this section
have many direct applications to various optimal control problems.
\end{remark}

Let us introduce a penalty function for problem \eqref{DiffInclProblem}. Define $X = W^d_{1,p}(0, T)$, and put
$$
  M = \Big\{ x \in X \Bigm| \dot{x}(t) \in F(x(t), t) \: \text{for a.e. } t \in [0, T] \Big\},
$$
and $A = \{ x \in X | x(0) = x_0 \}$ (note that this set is obviously closed). Then problem \eqref{DiffInclProblem} can
be rewritten as follows:
$$
  \min_{x \in X} \: \mathcal{I}(x) \quad \text{subject to} \quad x \in M \cap A.
$$
In order to introduce a penalty term $\varphi(x)$, denote $S = \{ \psi \in \mathbb{R}^d \mid |\psi| = 1 \}$, and for any
convex set $Y \subset \mathbb{R}^d$ and $\psi \in \mathbb{R}^d$ denote by 
$s(Y, \psi) = \sup_{y \in Y} \langle y, \psi \rangle$ the \textit{support function} of $Y$. 
By \cite[Theorem~13.1]{Rockafellar} a function $x \in X$ satisfies the differential 
inclusion $\dot{x}(t) \in F(x(t), t)$ iff
\begin{equation} \label{DiffInclViaSupportFunction}
  \langle \dot{x}(t), \psi \rangle \le s(F(x(t), t), \psi) \quad \forall \psi \in S \text{ for a.e. } t \in [0, T]
\end{equation}
or equivalently iff $h(x(t), \dot{x}(t), t) = 0$ for a.e. $t \in [0, T]$, where
\begin{equation} \label{DiffIncl_PenaltyTermIntegrand}
  h(x, z, t) = \max_{\psi \in S}\max\big\{ 0, \langle z, \psi \rangle - s(F(x, t), \psi) \big\}.
\end{equation}
Note that the maximum over all $\psi \in S$ in the definition of $h(x, z, t)$ is achieved, since the mapping 
$\psi \to s(F(x, t), \psi)$ is continuous, which, in turn, follows from the fact that $F(x, t)$ is a compact set.

Observe that $h(x, z, t) = \dist(z, F(x, t))$ for all $z, x \in \mathbb{R}^d$ and $t \in [0, T]$. Indeed,
for any $\psi \in S$ and $y \in F(x, t)$ one has 
$$
  \langle z, \psi \rangle - s(F(x, t), \psi) \le 
  \langle z, \psi \rangle - \langle y, \psi \rangle \le |z - y|.
$$
Taking the minimum over all $y \in F(x, t)$ one gets that
$$
  \langle z, \psi \rangle - s(F(x, t), \psi) \le \inf_{y \in F(x, t)} |z - y| =: \dist(z, F(x, t)).
$$
Consequently, $\max\{ 0, \langle z, \psi \rangle - s(F(x, t), \psi) \} \le \dist(z, F(x, t))$ by virtue of the fact
that $\dist(z, F(x, t)) \ge 0$. Hence taking the maximum over all $\psi \in S$ one obtains that 
$h(x, z, t) \le \dist(z, F(x, t))$. Clearly, this inequality turns into an equality when $z \in F(x, t)$.
Moreover, in the case $z \notin F(x, t)$ from the necessary conditions for a minimum \eqref{NessMinCond_OverConvexSet}
with $g(x) = |x|$ and $K = \{ z - y \mid y \in F(x, t) \}$ it follows that for $\psi^* = (z - y^*) / |z - y^*|$, where
$y^* \in F(x, t)$ is such that $\dist(z, F(x, t)) = |z - y^*|$, one has
$$
  \langle z - y, -\psi^* \rangle \le - | z - y^* | = - \dist(z, F(x, t)) \quad \forall y \in F(x, t)
$$
Taking the maximum over all $y \in F(x, t)$ one gets that 
$s(F(x, t), \psi^*) - \langle z, \psi^* \rangle \le - \dist(z, F(x, t))$, which obviously implies that
$h(x, z, t) \ge \dist(z, F(x, t))$. Thus, one has $h(x, z, t) = \dist(z, F(x, t))$ for all $z, x \in \mathbb{R}^d$ and
$t \in [0, T]$. 

Now one can formally define
$$
  \varphi(x) = \| h(x(\cdot), \dot{x}(\cdot), \cdot) \|_p = \left( \int_0^T 
  \max_{\psi \in S}\max\big\{ 0, \langle \psi, \dot{x}(t) \rangle - s(F(x(t), t), \psi) \big\}^p dt
  \right)^{\frac{1}{p}}.
$$
Clearly, $M = \{ x \in X \mid \varphi(x) = 0 \}$, which implies that one can consider the penalised problem
$$
  \min_{x \in X} \: \Phi_{\lambda}(x) = \mathcal{I}(x) + \lambda \varphi(x) \quad
  \text{subject to} \quad x \in A.
$$
Our aim is to provide sufficient conditions for the penalty function $\Phi_{\lambda}$ to be completely exact. For the
sake of simplicity, below we analyse only the simplest case when the support function $s(F(x, t), \psi)$ is
differentiable in $x$ for all $x \in \mathbb{R}^d$, $\psi \in S$, and $t \in [0, T]$. 

Before we can proceed to the theorem on the exactness of $\Phi_{\lambda}$, we need to obtain an auxiliary result on the
differentiability of the penalty term $\varphi$. Denote by $\psi^*(x, z, t)$ a vector $\psi \in S$ at which the maximum
in the definition of $h(x, z, t)$ is attained in the case $h(x, z, t) > 0$, and define $\psi^*(x, z, t) = \psi_0$
otherwise, where $\psi_0 \in S$ is a fixed vector. Note that in the case $h(x, z, t) > 0$ such $\psi^*(x, z, t)$ is
unique. Indeed, if the maximum is attained for $\psi = \psi_1 \in S$ and $\psi = \psi_2 \in S$ with $\psi_1 \ne \psi_2$,
then by applying the fact that the function $h_0(\psi) = \langle z, \psi \rangle - s(F(x, t), \psi)$ is concave one
obtains that $h_0(\xi) \ge 0.5 h_0(\psi_1) + 0.5 h_0(\psi_2)$, where 
$\xi = 0.5 \psi_1 + 0.5 \psi_2$. Note that $h_0(\psi_1) = h_0(\psi_2) = h(x, z, t)$ by the fact that $h(x, z, t) > 0$;
furthermore, $|\xi| < 1$, since the space $\mathbb{R}^d$ endowed with the Euclidean norm is strictly convex. Hence
taking into account the fact that the function $h_0$ is positively homogeneous of degree one we obtain that
$$
  h(x, z, t) \ge h_0\left( \frac{\xi}{|\xi|} \right) = 
  \frac{1}{|\xi|} h_0(\xi) > h_0(\xi) \ge 0.5 h_0(\psi_1) + 0.5 h_0(\psi_2) = h(x, z, t),
$$
which is impossible. Thus, $\psi^*(x, z, t)$ is well-defined in the case $h(x, z, t) > 0$.

\begin{proposition} \label{Prp_DiffIncl_PenTermDerivative}
Let the multifunction $F(x, t)$ be continuous on $\mathbb{R}^d \times [0, T]$, its support function $s(F(x, t), \psi)$
be differentiable in $x$, and the function $(x, t, \psi) \mapsto \nabla_x s(F(x, t), \psi)$ be continuous on
$\mathbb{R}^d \times [0, T] \times S$. Then the penalty term $\varphi(x)$ is correctly defined and finite on $X$,
G\^{a}teaux differentiable at every point $x \in W^d_{1,p}(0, T)$ such that $\varphi(x) > 0$, and
$$
  \varphi'(x)[v] = \frac{1}{\varphi(x)^{p - 1}} 
  \int_0^T h(x(t), \dot{x}(t), t)^{p - 1}
  \Big( \langle \psi^*(x, t), \dot{v}(t) \rangle
  - \left\langle \nabla_x s(F(x(t), t), \psi^*(x, t)), v(t) \right\rangle \Big) \, dt,
$$
where $\psi^*(x, t) = \psi^*(x(t), \dot{x}(t), t)$ for a.e. $t \in [0, T]$.
\end{proposition}

\begin{proof}
From the continuity of $F(x, t)$ it obviously follows that the function $h(x, z, t)$ is continuous
(see~\eqref{DiffIncl_PenaltyTermIntegrand}), which implies that the function $h(x(\cdot), \dot{x}(\cdot), \cdot)$ is
measurable for any $x \in W^d_{1, p}(0, T)$. Hence by applying the inequality
$$
  |h(x, z, t)|^p \le \big| |z| + \max_{\psi \in S} |s(F(x, t), \psi)| \big|^p \le 2^p |z|^p 
  + 2^p \max_{\psi \in S} |s(F(x, t), \psi)|^p
$$
one gets that the penalty term $\varphi(x)$ is correctly defined and finite on $X$. Let us check that the functional 
$$
  \varphi_0(x) = \varphi(x)^p = \int_0^T h(x(t), \dot{x}(t), t)^p \, dt
$$
is G\^{a}teaux differentiable and compute its derivative. Then by applying the chain rule one obtains the required
result.

Bearing in mind the facts that the support function $s(F(x, t), \psi)$ is differentiable in $x$, and 
the function $(x, t, \psi) \mapsto \nabla_x s(F(x, t), \psi)$ is continuous, and utilising a generalisation of 
the Danskin-Demyanov theorem (see~\cite[Theorem~4.13]{BonnansShapiro}) one obtains that the function
$h^p(x, z, t)$ is G\^{a}teaux differentiable in $x$ and $z$ at any point $(x, z, t)$ such that $h(x, z, t) > 0$, and
\begin{equation} \label{DistToInclGrad}
\begin{split}
  \nabla_x h^p(x, z, t) &= - p h(x, z, t)^{p - 1} \nabla_x s(F(x, t), \psi^*(x, z, t)), \\
  \nabla_z h^p(x, z, t) &= p h(x, z, t)^{p - 1} \psi^*(x, z, t).
\end{split}
\end{equation}
Let us consider the case $h(x, z, t) = 0$. Note that the function $x \mapsto s(F(x, t), \psi)$ is locally Lipschitz
continuous with the same Lipschitz constant for all $t \in [0, T]$ and $\psi \in S$, since its derivative in $x$ is
continuous. Hence, as it is easy to check, the function  $(x, z) \to h(x, z, t)$ is locally Lipschitz continuous with
the same Lipschitz constant for any $t \in [0, T]$. Utilising this fact and the inequality $p > 1$ one obtains that if
$(x, z, t)$ is such that $h(x, z, t) = 0$, then the function $h^p(x, z, t)$ is G\^{a}teaux differentiable in $x$ and $z$
as well, and $\nabla_x h^p(x, z, t) = 0$ and $\nabla_z h^p(x, z, t) = 0$, since
$$
  \big| h^p(x + \Delta x, z + \Delta z, t) - h^p(x, z, t) | \le L^p \big( |\Delta x| + |\Delta z| \big)^p
$$ 
for any $\Delta x$ and $\Delta z$ in a neighbourhood of zero, and for some $L > 0$. Moreover, with the use of the facts
that the mapping $\psi^*(\cdot)$ is continuous on the open set $\{ (x, z, t) \mid h(x, z, t) > 0 \}$ by
\cite[Proposition~4.4]{BonnansShapiro} and $|\psi^*(\cdot)| \equiv 1$ one can easily check that the functions 
$\nabla_x h^p(\cdot)$ and $\nabla_z h^p(\cdot)$ are continuous on $\mathbb{R}^d \times \mathbb{R}^d \times [0, T]$.

Fix $x, v \in W^d_{1, p}(0, T)$. By the mean value theorem for any
$\alpha \in (0, 1]$ and for a.e. $t \in (0, T)$ there exists $\alpha(t) \in (0, \alpha)$ such that
\begin{multline*}
  \frac{1}{\alpha} \Big( h^p(x(t) + \alpha v(t), \dot{x}(t) + \alpha \dot{v}(t), t) - h^p(x(t), \dot{x}(t), t) \Big) \\
  = \langle \nabla_x h^p(x(t) + \alpha(t) v(t), \dot{x}(t) + \alpha(t) \dot{v}(t), t), v(t) \rangle
  + \langle \nabla_z h^p(x(t) + \alpha(t) v(t), \dot{x}(t) + \alpha(t) \dot{v}(t), t), \dot{v}(t) \rangle
\end{multline*}
for a.e. $t \in [0, T]$. The right hand side of this equality converges to
$$
  \langle \nabla_x h^p(x(t), \dot{x}(t), t), v(t) \rangle
  + \langle \nabla_z h^p(x(t), \dot{x}(t), t), \dot{v}(t) \rangle
$$
as $\alpha \to + 0$ for a.e. $t \in [0, T]$ due to the continuity of $\nabla_x h^p(\cdot)$ and $\nabla_z h^p(\cdot)$.
Taking into account the obvious inequalities
\begin{gather*}
  |\nabla_x h^p(x, z, t)| \le p |h(x, z, t)|^{p - 1} | \nabla_x s(F(x, t), \psi^*(x, z, t)) |, \qquad
  |\nabla_z h^p(x, z, t)| \le p |h(x, z, t)|^{p - 1}, \\
  |h(x, z, t)|^{p - 1} \le \big| |\langle z, \psi^*(x, z, t) \rangle| + |s(F(x, t), \psi^*(x, z, t))| \big|^{p - 1}
  \le 2^{p - 1} \big( |z|^{p - 1} + \max_{\psi \in S} |s(F(x, t), \psi)|^{p - 1} \big),
\end{gather*}
and the fact that $\| x \|_{\infty} \le C \| x \|_{1, d}$ for some $C > 0$ (see Remark~\ref{Remark_EquivalentNorms}) 
one obtains that there exist $C_1, C_2 > 0$ such that
$$
  |\langle \nabla_x h^p(x(t) + \alpha(t) v(t), \dot{x}(t) + \alpha(t) \dot{v}(t), t), v(t) \rangle|
  \le (C_1 |\dot{x}(t)|^{p - 1} + C_1 |\dot{v}(t)|^{p - 1} + C_2) |v(t)|,
$$
and
$$
  |\langle \nabla_z h^p(x(t) + \alpha(t) v(t), \dot{x}(t) + \alpha(t) \dot{v}(t), t), \dot{v}(t) \rangle|
  \le (C_1 |\dot{x}(t)|^{p - 1} + C_1 |\dot{v}(t)|^{p - 1} + C_2) |\dot{v}(t)|
$$
for a.e. $t \in [0, T]$ and for any $\alpha \in (0, 1]$. Note that the right-hand sides of these inequalities belong to
$L^1(0, T)$. Therefore, by applying Lebesgue's dominated convergence theorem one gets that
$$
  \lim_{\alpha \to + 0} \frac{\varphi_0(x + \alpha v) - \varphi_0(x)}{\alpha}
  = \int_0^T \Big( \langle \nabla_x h^p(x(t), \dot{x}(t), t), v(t) \rangle
  + \langle \nabla_z h^p(x(t), \dot{x}(t), t), \dot{v}(t) \rangle \Big) \, dt,
$$
which along with \eqref{DistToInclGrad} completes the proof.
\end{proof}

Let $\mathcal{I}^*$ be the optimal value of problem \eqref{DiffInclProblem}.

\begin{theorem}
Let the following assumptions be valid:
\begin{enumerate}
\item{$\zeta$ is locally Lipschitz continuous, $\theta(x, t)$ is continuous, differentiable in $x$, and its derivative
in $x$ is continuous;
\label{Assumpt_DIP_ObjectiveFunction}}

\item{the multifunction $F(x, t)$ is continuous, its support function $s(F(x, t), \psi)$ is differentiable in $x$, and
the function $(x, t, \psi) \mapsto \nabla_x s(F(x, t), \psi)$ is continuous;
\label{Assumpt_DIP_DiffIncl}}

\item{there exists a globally optimal solution of problem \eqref{DiffInclProblem};}

\item{there exist $\lambda_0 > 0$, $\delta > 0$, and $c > \mathcal{I}^*$ such that the set 
$S_{\lambda_0}(c) \cap \Omega_{\delta}$ is bounded in $W^d_{1, p}(0, T)$, and the function $\Phi_{\lambda_0}(x)$ is
bounded below on $A$.
\label{Assumpt_DIP_SublevelBound}}
\end{enumerate}
Then there exists $\lambda^* \ge 0$ such that for any $\lambda \ge \lambda^*$ the penalty function $\Phi_{\lambda}$ for
problem~\eqref{DiffInclProblem} is completely exact on $S_{\lambda}(c)$.
\end{theorem}

\begin{proof}
Note that the functional $\mathcal{I}(x)$ is Lipschitz continuous on any bounded subset of $W^d_{1, p}(0, T)$ by 
Proposition~\ref{Prp_LipschitzConditions}, Remark~\ref{Remark_EquivalentNorms}, and the fact that $\zeta$ is locally
Lipschitz continuous. Furthermore, taking into account the obvious estimate
$$
  |h(x, z, t)|^p \le \max_{\psi \in S} \big| \langle z, \psi \rangle - s(F(x, t), \psi) \big|^p
  \le 2^p \big( |z|^p + \max_{\psi \in S} |s(F(x, t), \psi)|^p \big),
$$
and utilising Vitali's convergence theorem (see~\cite[Theorem~III.6.15]{DunfordSchwartz}) one can easily check that 
the operator $x \to h(x(\cdot), \dot{x}(\cdot), \cdot)$ continuously maps $W^d_{1, p}(0, T)$ to $L^p(0, T)$ 
(cf.~\cite[Theorem~5.1]{Precup}), i.e. the penalty term $\varphi$ is continuous. Thus, by
Theorem~\ref{THEOREM_COMPLETEEXACTNESS} it remains to check that for any bounded
set $K \subset A$ there exists $a > 0$ such that $\varphi^{\downarrow}_A(x) \le - a$ for any 
$x \in K \setminus \Omega$.

By applying the technique of ``transition into the space of derivatives'' as in the proof of
Theorem~\ref{Thrm_FreeEndPointProblem} one obtains that it is sufficient to check that for any bounded 
set $Z \subset L^d_p(0, T)$ there exists $a > 0$ such that $\gamma^{\downarrow}(z) \le - a$ for all $z \in Z$ with
$\gamma(z) > 0$, where
$\gamma(z) = \varphi(Jz)$ and $(J z)(t) = x_0 + \int_0^t z(\tau) \, d \tau$.
Let $Z \subset L^d_p(0, T)$ be a bounded set. Utilising Proposition~\ref{Prp_DiffIncl_PenTermDerivative} and
integrating by parts one gets that for any $z \in Z$ such that $\gamma(z) > 0$ the functional $\gamma(\cdot)$ is
G\^{a}teaux differentiable at $z$, and
$$
  \gamma'(z) [v]
  = \int_0^T \Big\langle w(z)(t) \psi^*(x, t) 
  - \int_t^T w(z)(\tau) \nabla_x s\big( F(x(\tau), \tau), \psi^*(x, \tau) \big) \, d \tau,
  v(t) \Big\rangle \, dt,
$$
where $x = Jz$ and $w(z)(t) = \gamma(z)^{1 - p} h(x(t), z(t), t)^{p - 1}$. Therefore, for any such $z$ one has
$\gamma^{\downarrow}(z) \le - \| \gamma'(z) \| = - \| H(z) \|_{p'} \le 0$, where
\begin{equation} \label{IntegralOperator_DiffIncl}
  H(z)(t) = 
  w(z)(t) \psi^*(x, t) - \int_t^T w(z)(\tau) \nabla_x s\big( F(x(\tau), \tau), \psi^*(x, \tau) \big) \, d \tau.
\end{equation}
Arguing by reductio ad absurdum, suppose that there exists a sequence $\{ z_n \} \subset Z$ such that 
$\gamma(z_n) > 0$ for all $n \in \mathbb{N}$, and $\gamma^{\downarrow}(z_n) \to 0$ as $n \to \infty$. Then 
$\| H(z_n) \|_{p'} \to 0$ as $n \to \infty$ as well. 

For any $n \in \mathbb{N}$ define $\psi_n(\cdot) = \psi^*(x_n, \cdot) = \psi^*( x_n(\cdot), z_n(\cdot), \cdot)$, where 
$x_n = J z_n$. Let us check that these functions are measurable. From the continuity of the multifunction $F(x, t)$ it
follows that the function $h(x, z, t)$ is continuous as well. Therefore, the function 
$h(x_n(\cdot), z_n(\cdot), \cdot)$ is measurable, which implies that the set 
$E_n = \{ t \in [0, T] \mid h(x_n(t), z_n(t), t) > 0 \}$ is measurable. As was pointed out above, the function
$\psi^*(x, z, t)$ is continuous on the open set $V = \{ (x, z, t) \mid h(x, z, t) > 0 \}$ by
\cite[Proposition~4.4]{BonnansShapiro}. Consequently, the function $E_n \ni t \mapsto \psi^*(x_n(t), z_n(t), t)$ is
measurable as the composition of the restriction of $\psi^*(x, z, t)$ to $V$ and the measurable mapping
$E_n \ni t \to (x_n(t), z_n(t), t)$. Hence one obtains that
$$
  \psi_n(t) = \begin{cases}
    \psi^*(x_n(t), z_n(t), t), & \text{if } t \in E_n, \\
    \psi_0, & \text{if } t \in [0, T] \setminus E_n.
  \end{cases}
$$
is measurable (recall that $\psi_0 \in S$ is a fixed vector; see the discussion before
Proposition~\ref{Prp_DiffIncl_PenTermDerivative}).

Recall that $\| H(z_n) \|_{p'} \to 0$ as $n \to \infty$. Hence and from the fact that by definition 
$|\psi_n(\cdot)| \equiv 1$ it follows that $\| \langle \psi_n, H(z_n) \rangle \|_{p'} \to 0$ as $n \to \infty$ as well.
On the other hand, from the equalities $\| w(z_n) \|_{p'} = 1$ and $|\psi_n(\cdot)| \equiv 1$ it follows that
$$
  \| \langle \psi_n, H(z_n) \rangle \|_{p'} = \| (I - \mathcal{K}_{y_n}) w(z_n) \|_{p'} 
  \ge \frac{1}{\| (I - \mathcal{K}_{y_n})^{-1} \|}
$$
(see~\eqref{IntegralOperator_DiffIncl}), where $I$ is the identity operator and
$$
  \Big( \mathcal{K}_{y_n} h \Big)(t) = \int_t^T y_n(t, s) h(s) \, ds 
  \quad \forall h \in L^{p'}(0, T), \quad
  y_n(t, s) = \langle \nabla_x s( F(x_n(s), s), \psi_n(s)), \psi_n(t) \rangle.
$$
Observe that by Lemma~\ref{Lemma_BoundedResolvent} one has 
$\sup_{n \in \mathbb{N}} \| (I - \mathcal{K}_{y_n})^{-1} \| < + \infty$ due to the inequality
$$
  |y_n(t, s)| \le \max_{\psi \in S} \big| \nabla_x s( F(x_n(s), s), \psi ) \big|
  \quad \text{for a.e. } t, s \in [0, T]
$$
and the fact that the sequence $\{ x_n \}$ is bounded in $L_{\infty}^d(0, T)$, which in turn follows from the
boundedness of the set $Z$ and Remark~\ref{Remark_EquivalentNorms}. Thus, 
$\inf_{n \in \mathbb{N}} \| \langle \psi_n, H(z_n) \rangle \|_{p'} > 0$, which contradicts the fact that 
$\| \langle \psi_n, H(z_n) \rangle \|_{p'} \to 0$ as $n \to \infty$. Therefore, there exists $a > 0$ such that
$\gamma^{\downarrow}(z) \le - a$ for any $z \in Z$ with  $\gamma(z) > 0$, and the proof is complete.
\end{proof}

\begin{remark}
Note that under the assumptions \ref{Assumpt_DIP_ObjectiveFunction}, \ref{Assumpt_DIP_DiffIncl}, and
\ref{Assumpt_DIP_SublevelBound} of the theorem above a globally optimal solution of problem \eqref{DiffInclProblem}
exists iff there exists a feasible point of this problem, i.e. there exists an absolutely continuous solution of the
differential inclusion $\dot{x} \in F(x, t)$ starting at $x_0$ and defined on $[0, T]$. Indeed, from the existence
of a feasible point of problem \eqref{DiffInclProblem} and the inequality $c > \mathcal{I}^*$ it follows that the
sublevel set $\{ x \in \Omega \mid \mathcal{I}(x) \le c \} \subseteq S_{\lambda_0}(c) \cap \Omega_{\delta}$ is nonempty
and bounded. Therefore, there exists a bounded sequence $\{ x_n \} \subset \Omega$ such that 
$\mathcal{I}(x_n) \to \mathcal{I}^*$ as $n \to \infty$. Taking into account the reflexivity of 
the space $W^d_{1, p}(0, T)$ one obtains that there exists a subsequence $\{ x_{n_k} \}$ weakly converging to some 
$x^* \in W^d_{1, p}(0, T)$. By Remark~\ref{Remark_EquivalentNorms} one can suppose that $x_{n_k}$ converges to $x^*$
uniformly on $[0, T]$, which, as it is easily seen, implies that $\mathcal{I}(x^*) = \lim_{k \to \infty}
\mathcal{I}(x_{n_k}) = \mathcal{I}^*$ and 
$x^* \in A$. Thus, it remains to check that $x^*$ is a solution of the differential inclusion $\dot{x} \in F(x, t)$. 

From \eqref{DiffInclViaSupportFunction} and the fact that $x_n \in \Omega$ it follows that
$$
  \int_0^T \langle \dot{x}_{n_k}(t), \psi(t) \rangle \, dt \le
  \int_0^T s\big( F(x_{n_k}(t), t), \psi(t) \big) \, dt \quad \forall \psi \in L_{p'}^d(0, T).
$$
Passing to the limit as $k \to \infty$ with the use of Lebesgue's dominated convergence theorem, and the facts that
$x_{n_k}$ converges to $x^*$ uniformly in $[0, T]$ and the compact-valued multifunction $F(x, t)$ is continuous one
obtains that
$$
  \int_0^T s\big( F(x^*(t), t) - \dot{x}^*(t), \psi(t) \big) \, dt \ge 0 \quad \forall \psi \in L_{p'}^d(0, T).
$$
By Filippov's theorem (see, e.g. \cite{Filippov} or \cite[Theorem~8.2.10]{AubinFrankowska}) there exists a measurable
selection $y(t)$ of $F(x^*(t), t)$ such that for almost every $t \in [0, T]$ one has 
$s( F(x^*(t), t) - \dot{x}^*(t), \psi(t)) = \langle y(t) - \dot{x}^*(t), \psi(t) \rangle$.
Consequently, one has
$$
  \sup_{y \in \mathcal{F}(x^*)} \int_0^T \langle y(t) - \dot{x}^*(t), \psi(t) \rangle \, dt \ge 0 
  \quad \forall \psi \in L_{p'}^d(0, T),
$$
where $\mathcal{F}(x^*)$ is the set of all measurable selections of the multifunction $F(x^*(\cdot), \cdot)$. With the
use of the facts that the set-valued map $F$ is continuous and its values are compact and convex one can check that 
the set $\mathcal{F}(x^*)$ is convex, closed and bounded in $L_p^d(0,T)$, i.e. it is weakly compact in $L_p^d(0, T)$.
Hence by applying the separation theorem and the inequality above one gets that $0 \in \mathcal{F}(x^*) - \dot{x}^*$ or,
equivalently, $\dot{x}^*(t) \in F(x^*(t), t)$ for a.e. $t \in [0, T]$.

Let us finally note that a feasible point of problem \eqref{DiffInclProblem} exists, for example, if there exist 
$C > 0$ and a.e. nonnegative function $\omega \in L^p(0, T)$ such that $|v| \le C |x| + \omega(t)$ for all 
$v \in F(x, t)$, $x \in \mathbb{R}^d$ and a.e. $t \in (0, T)$ (this result follows directly from 
the Gr\"{o}nwall-Bellman inequality).
\end{remark}

\section{Conclusions}

In the first paper of our two-part study we strengthened some existing results on exact penalty functions for
optimisation problems in infinite dimensional spaces and applied them to free-endpoint optimal control problems. We
obtained simple and verifiable sufficient conditions for the complete exactness of penalty functions for such problems
with the use of a number of auxiliary results on integral functionals and Nemytskii operators proved in this paper. Our
results allow one to reduce free-endpoint point optimal control problems (including such problems involving
differential inclusions) to equivalent variational problems and, thus, apply numerical methods of nonsmooth optimisation
to such optimal control problems.

The equivalent variational problems obtained with the use of exact penalty functions can be discretised directly and
then solved with the use of many modern methods of nonsmooth optimisation
(see~\cite{KarmitsaBagriovMakela2012,BagirovKarmitsaMakela_book} for a comparative analysis of existing nonsmooth
optimisation software). Alternatively, they can be solved via continuous methods, such as the method of
hypodifferential descent (see~\cite{FominyhKarelin2018} for more details and numerical examples) and an infinite
dimensional version of bundle methods \cite{Outrata83,OutrataSchindler,Outrata88}. Furthermore, one can solve the
equivalent variational problems via smooth optimisation methods (both continuous and based on discretisation) by
applying smoothing approximations of nonsmooth penalty functions \cite{Pinar,Liu,LiuzziLucidi,Lian,Dolgopolik_ExPen_II}
or by replacing these problems with equivalent problems of minimising the smooth penalty function proposed by Huyer
and Neumaier \cite{Dolgopolik_ExPen_II,HuyerNeumaier,WangMaZhou,Dolgopolik_SmoothExPen}.

Thus, our results pave the way for a comparative analysis of various nonsmooth optimisation methods for solving
optimal control problems, as well as for a comparative analysis of smooth and nonsmooth approaches to the solution of
optimal control problems. Moreover, our results can be extended to the case of \textit{nonsmooth} optimal control
problems and utilised to develop new numerical methods for solving such problems.

The second paper of our two-part study will be devoted to the analysis of exact penalty functions for optimal control
problems with terminal and pointwise state constraints, including optimal control problems for linear evolution
equations in Hilbert spaces.

\bibliographystyle{abbrv}  
\bibliography{ExactPen_OptControl}

\section*{Appendix A. The Proof of Theorem~\ref{THEOREM_COMPLETEEXACTNESS}}

For the sake of completeness, let us present an almost self-contained proof of Theorem~\ref{THEOREM_COMPLETEEXACTNESS},
although some parts of this theorem can be found in \cite{Dolgopolik_ExPen_I,Dolgopolik_ExPen_II}. The only result
we will use is Ekeland's variational principle \cite{Ekeland}. Let us recall that to apply this principle one must
suppose that the function under consideration is defined on a complete metric space, l.s.c., and bounded below. That is
why we must suppose that $X$ is complete, $A$ is closed, and both functions $\mathcal{I}$ and $\varphi$ are lower
semi-continuous.

\noindent\textbf{Proof of part 1.} Arguing by reductio ad absurdum, suppose that the optimal values of the problems
$(\mathcal{P})$ and \eqref{PenalizedProblem} do not coincide for any $\lambda \ge 0$ (note that if they coincide for
some $\lambda^* \ge 0$, then they coincide for all $\lambda \ge \lambda^*$ due to the fact that $\Phi_\lambda(x)$ is
nondecreasing in $\lambda$). Recall that the penalty term $\varphi$ is nonnegative and $\varphi(x) = 0$ iff $x \in M$.
Consequently, 
\begin{equation} \label{PenaltyFuncOnFeasibleSet}
  \Phi_{\lambda}(x) = \mathcal{I}(x) \quad \forall x \in \Omega = M \cap A,
\end{equation}
which implies that $\inf_{x \in A} \Phi_{\lambda}(x) < \mathcal{I}^* = \inf_{x \in \Omega} \mathcal{I}(x)$ for any
$\lambda \ge 0$ (note also that $\inf_{x \in A} \Phi_{\lambda}(x) > - \infty$ for any $\lambda \ge \lambda_0$ due to
the fact that $\Phi_{\lambda}$ is nondecreasing in $\lambda$). Hence, in particular, for any $n \in \mathbb{N}$ there
exists $x_n \in A$ such that $\Phi_n(x_n) < \mathcal{I}^*$. Define 
$\varepsilon_n = \Phi_n(x_n) - \inf_{x \in A} \Phi_n(x) + 1/n$. By applying Ekeland's variational principle one obtains
that for any $n \in \mathbb{N}$, $n \ge \lambda_0$, and $t > 0$ there exists $y_n \in A$ such that 
$\Phi_n(y_n) \le \Phi_n(x_n)$, and the following inequalities hold true:
$$
  d(y_n, x_n) \le t, \quad
  \Phi_n(y) - \Phi_n(y_n) > - \frac{\varepsilon_n}{t} d(y, y_n) \quad \forall y \in A \setminus \{ y_n \}.
$$
Setting $t = \varepsilon_n$, dividing the last inequality by $d(y, y_n)$, and passing to the limit inferior as 
$y \to y_n$, $y \in A$ one gets that 
\begin{equation} \label{FermatApproxRule_ExPen}
  (\Phi_n)^{\downarrow}_A (y_n) \ge -1 \quad \forall n \in \mathbb{N} \colon n \ge \lambda_0
\end{equation}
(note that if $y_n$ is an isolated point of $A$, then by definition $(\Phi_n)^{\downarrow}_A (y_n) = + \infty$). 
From the facts that $\Phi_n(y_n) \le \Phi_n(x_n) < \mathcal{I}^* < c$ and $\Phi_n(x) = \mathcal{I}(x)$ for any 
$x \in \Omega$ it follows that $y_n \in S_n(c)$ and $y_n \notin \Omega$ for any $n \in \mathbb{N}$. Observe also that
for any $n \ge \lambda_0$, $m \in \mathbb{N}$, and $x \in A$ such that 
$x \notin \Omega_{\delta} = \{ x \in A \mid \varphi(x) < \delta \}$ one has
$$
  \Phi_{n + m}(x) = \mathcal{I}(x) + (n + m) \varphi(x) 
  = \Phi_n(x) + m \varphi(x) \ge \inf_{x \in A} \Phi_{\lambda_0}(x) + m \delta.
$$
Consequently, for any sufficiently large $n$ one has $\Phi_n(x) \ge \mathcal{I}^*$, provided 
$x \in A \setminus \Omega_{\delta}$, which implies that there exists $n_0 \ge \lambda_0$ such that 
$y_n \in \Omega_{\delta}$ for all $n \ge n_0$.

Thus, $y_n \in S_{\lambda_0}(c) \cap (\Omega_{\delta} \setminus \Omega)$ for any $n \ge n_0$ (here we used the fact
that $S_n(c) \subseteq S_{\lambda_0}(c)$ for any $n \ge \lambda_0$, since $\Phi_{\lambda}$ is nondecreasing in
$\lambda$). Therefore, $\varphi^{\downarrow}_A(y_n) \le -a$ for all $n \ge n_0$. By the definition of the rate of
steepest descent for any $n \ge n_0$ there exists a sequence $\{ y_n^k \} \subset A$, $k \in \mathbb{N}$, converging to
$y_n$ such that $\varphi(y_n^k) - \varphi(y_n) \le - 0.5 a d(y_n^k, y_n)$ for all $k \in \mathbb{N}$. Hence taking into
account the fact that the function $\mathcal{I}$ is Lipschitz continuous on an open set containing the set 
$S_{\lambda_0}(c) \cap \Omega_{\delta}$ with a Lipschitz constant $L \ge 0$ one obtains that for any
$n \ge n_0$ there exists $k(n) \in \mathbb{N}$ such that for all $k \ge k(n)$ one has
$$
  \Phi_n(y_n^k) - \Phi_n(y_n) 
  = \mathcal{I}(y_n^k) - \mathcal{I}(y_n) + n \big(\varphi(y_n^k) - \varphi(y_n) \big)
  \le L d(y_n^k, y_n) - \frac{n a}{2} d(y_n^k, y_n) = \left( L - \frac{n a}{2} \right) d(y_n^k, y_n).
$$
Dividing this inequality by $d(y_n^k, y_n)$, and passing to the limit inferior as $k \to + \infty$ one obtains that
$(\Phi_n)^{\downarrow}_A (y_n) \le L - 0.5 n a < - 1$ for any sufficiently large $n$, which contradicts
\eqref{FermatApproxRule_ExPen}. \qed

\noindent\textbf{Proof of part 2.} By the first part of the theorem there exists $\lambda^* \ge 0$ such that for any
$\lambda \ge \lambda^*$ one has $\inf_{x \in A} \Phi_{\lambda}(x) = \mathcal{I}^* = \inf_{x \in \Omega} \mathcal{I}(x)$.
Hence by applying \eqref{PenaltyFuncOnFeasibleSet} one gets that 
$\argmin_{x \in \Omega} \mathcal{I}(x) \subseteq \argmin_{x \in A} \Phi_{\lambda}(x)$ for all
$\lambda \ge \lambda^*$. On ther other hand, if $x \in A \setminus \Omega$, then $\varphi(x) > 0$, and for any 
$\lambda > \lambda^*$ one has $\Phi_{\lambda}(x) > \Phi_{\lambda^*}(x) \ge \mathcal{I}^*$. Therefore, for any 
$\lambda > \lambda^*$ one has $\argmin_{x \in A} \Phi_{\lambda}(x) \subset \Omega$, which with the use of
\eqref{PenaltyFuncOnFeasibleSet} implies that 
$\argmin_{x \in \Omega} \mathcal{I}(x) = \argmin_{x \in A} \Phi_{\lambda}(x)$. \qed

Before we proceed to the proof of the last two statements of Theorem~\ref{THEOREM_COMPLETEEXACTNESS}, let us first prove
two auxiliary lemmas. The first one is a modification of the main lemma from \cite{Ioffe}, while the second one is a
generalisation of \cite[Proposition~2.7]{Dolgopolik_ExPen_I}.

\begin{lemma} \label{Lemma_RetricSubregularity}
Let the assumptions of Theorem~\ref{THEOREM_COMPLETEEXACTNESS} be valid. Then for any 
$x_0 \in S_{\lambda_0}(c) \cap \Omega$ there exists $r > 0$ such that $\varphi(x) \ge a \dist(x, \Omega)$ for all 
$x \in B(x_0, r) \cap A$, where $B(x_0, r) = \{ y \in X \mid d(y, x_0) \le r \}$.
\end{lemma}

\begin{proof}
Fix $x_0 \in S_{\lambda_0}(c) \cap \Omega$. Note that $\varphi(x_0) = 0$, since $x_0 \in \Omega$. Due to the continuity
of $\varphi$ there exists $\eta > 0$ such that $\varphi(x) < \delta$ for any $x \in B(x_0, \eta)$, i.e.
$B(x_0, \eta) \subset \Omega_{\delta}$. Furthermore, one can choose $\eta > 0$ so small that $\mathcal{I}$ is Lipschitz
continuous on $B(x_0, \eta)$. Consequently, decreasing $\eta$ further if necessary, one can suppose that 
$B(x_0, \eta) \subset S_{\lambda_0}(c) \cap \Omega_{\delta}$, which implies that $\varphi^{\downarrow}(x) \le - a < 0$
for all $x \in B(x_0, \eta) \setminus \Omega$.

The continuity of $\varphi$ implies that there exists $r > 0$ such that 
\begin{equation} \label{PhiContinuity}
  \varphi(x) < \frac{\eta a}{4} \quad \forall x \in B(x_0, r).
\end{equation}
Moreover, one can obviously suppose that $r < \eta / 2$. Let $x \in B(x_0, r) \cap A$ be arbitrary. If $x \in \Omega$,
then $\varphi(x) = 0$, and the inequality $\varphi(x) \ge a \dist(x, \Omega)$ is satisfied. Suppose that 
$x \notin \Omega$. Denote $\varepsilon = \varphi(x) > 0$, and choose an arbitrary $t \in (a/2, a)$. By applying
Ekeland's variational principle one obtains that there exists $y \in A$ such that $\varphi(y) \le \varphi(x)$, and
\begin{equation} \label{EkelandVarPrinciple}
  d(y, x) \le \frac{\varepsilon}{t}, \quad
  \varphi(z) + t d(z, y) > \varphi(y) \quad \forall z \in A \setminus \{ y \}.
\end{equation}
Let us check that $\varphi(y) = 0$. Indeed, if $\varphi(y) > 0$, then taking into account \eqref{PhiContinuity} one
obtains that $y \in B(x_0, \eta) \setminus \Omega$, since
$$
  x \in B(x_0, r) \subset B\left(x_0, \frac{\eta}{2} \right), \quad
  d(y, x) \le \frac{\varepsilon}{t} < \frac{2 \varphi(x)}{a} < \frac{\eta}{2}.
$$
Hence $\varphi^{\downarrow}(y) \le - a$, which implies that there exists $z \in A$, $z \ne y$, such that
$\varphi(z) - \varphi(y) \le - t d(z, y)$. Therefore $\varphi(z) + t d(z, y) \le \varphi(y)$, which contradicts the
second inequality in \eqref{EkelandVarPrinciple}. Thus, $\varphi(y) = 0$, i.e. $y \in \Omega$. Now, by applying the
first inequality in \eqref{EkelandVarPrinciple} one obtains that
$$
  \dist(x, \Omega) \le d(x, y) \le \frac{\varepsilon}{t} = \frac{\varphi(x)}{t}
$$
or equivalently $\varphi(x) \ge t \dist(x, \Omega)$. Hence taking into account the fact that $t \in (a/2, a)$ was chosen
arbitrarily one gets the required result.
\end{proof}

\begin{lemma} \label{Lemma_LipschitzEstimate}
Let the assumptions of Theorem~\ref{THEOREM_COMPLETEEXACTNESS} be valid, and $L$ be a Lipschitz constant of
$\mathcal{I}$ on an open set containing the set $S_{\lambda_0}(c) \cap \Omega_{\delta}$. Suppose also that 
$x^* \in S_{\lambda_0}(c) \cap \Omega$ is an inf-stationary point of $\mathcal{I}$ on $\Omega$. Then for any $L' > L$
there exists $r > 0$ such that $\mathcal{I}(x) - \mathcal{I}(x^*) \ge - L'\dist(x, \Omega) - (L' - L) d(x, x^*)$ for
all $x \in B(x^*, r)$.
\end{lemma}

\begin{proof}
Choose $L' > L$. By the definition of inf-stationary point there exists $r_0 > 0$ such that
$\mathcal{I}(x) - \mathcal{I}(x^*) \ge -(L' - L) d(x, x^*)$ for all $x \in B(x^*, r_0) \cap \Omega$. Decreasing $r_0$,
if necessary, one can suppose that $\mathcal{I}$ is Lipschitz continuous on $B(x^*, r_0)$ with Lipschitz constant $L$.

Put $r = r_0 / 2$, and fix an arbitrary $x \in B(x^*, r)$. By definition there exists a sequence 
$\{ x_n \} \subset \Omega$ such that $d(x, x_n) \to \dist(x, \Omega)$ as $n \to \infty$ and 
$d(x, x_n) \le d(x, x^*) \le r$. Hence, in particular, one has
\begin{equation} \label{MinSeqDistOmega}
  d(x^*, x_n) \le d(x^*, x) + d(x, x_n) \le r + r = 2 r \le r_0,
\end{equation}
i.e. $\{ x_n \} \subset B(x^*, r_0) \cap \Omega$ for all $n \in \mathbb{N}$. Therefore,
$$
  \mathcal{I}(x) - \mathcal{I}(x^*) = \mathcal{I}(x) - \mathcal{I}(x_n) + \mathcal{I}(x_n) - \mathcal{I}(x^*)
  \ge - L d(x, x_n) - (L' - L) d(x_n, x^*)
  \ge - L' d(x, x_n) - (L' - L) d(x, x^*)
$$
for any $n \in \mathbb{N}$, where the last inequality follows from \eqref{MinSeqDistOmega}. Passing to the limit as 
$n \to \infty$ we arrive at the required result.
\end{proof}

\begin{remark} \label{Remark_LipschitzEstimate}
Note that if $x^*$ is a point of local minimum of $\mathcal{I}$ on $\Omega$, then in the lemma above one can obviously
set $L' = L$, and check that for any $x \in B(x^*, r)$ one has 
$\mathcal{I}(x) - \mathcal{I}(x^*) \ge - L\dist(x, \Omega)$ (see \cite[Proposition~2.7]{Dolgopolik_ExPen_I}).
\end{remark}

Now we are ready to prove the last two statements of Theorem~\ref{THEOREM_COMPLETEEXACTNESS}.

\noindent\textbf{Proof of part 3.} At first, note that without loss of generality one can suppose that 
$\delta = + \infty$. Indeed, denote $\eta = \inf_{x \in A} \Phi_{\lambda_0}(x) > - \infty$. Then for any 
$x \in A \setminus \Omega_{\delta}$ and $\lambda > \widehat{\lambda} := \lambda_0 + (c - \eta) / \delta$ one has
$$
  \Phi_{\lambda}(x) = \Phi_{\lambda_0}(x) + (\lambda - \lambda_0) \varphi(x) 
  \ge \eta + (\lambda - \lambda_0) \delta \ge c, 
$$
which implies that $S_{\lambda}(c) \subseteq S_{\lambda_0}(c) \cap \Omega_{\delta}$ for any 
$\lambda > \widehat{\lambda}$. Thus, increasing if necessary $\lambda_0$ one can suppose that $\delta = + \infty$,
i.e. one can replace $S_{\lambda_0}(c) \cap \Omega_{\delta}$ with $S_{\lambda_0}(c)$. Note also that 
\begin{equation} \label{ShrinkingSublevelSets}
  S_{\lambda}(c) \subseteq S_{\lambda_0}(c) \quad \forall \lambda \ge \lambda_0
\end{equation}
by virtue of the fact that $\Phi_{\lambda}$ is non-decreasing in $\lambda$.

Let $L > 0$ be a Lipschitz constant of $\mathcal{I}$ on an open set $V$ containing the set $S_{\lambda_0}(c)$. By our
assumption for any $x \in S_{\lambda_0}(c) \setminus \Omega$ one has $\varphi^{\downarrow}_A(x) \le - a$. Hence by 
the definition of the rate of steepest descent there exists a sequence $\{ x_n \} \subset A$ converging to $x$ and such
that
$$
  \liminf_{n \to \infty} \frac{\varphi(x_n) - \varphi(x)}{d(x_n, x)} \le - a.
$$
One can obviously suppose that $\{ x_n \} \subset V$. Therefore for any $\lambda > 0$ one has
\begin{align*}
  (\Phi_{\lambda})^{\downarrow}_A(x) 
  \le \liminf_{n \to \infty} \frac{\Phi_{\lambda}(x_n) - \Phi_{\lambda}(x)}{d(x_n, x)}
  &= \liminf_{n \to \infty} \frac{\mathcal{I}(x_n) - \mathcal{I}(x) + \lambda (\varphi(x_n) - \varphi(x))}{d(x_n, x)} \\
  &\le L + \lambda \liminf_{n \to \infty} \frac{\varphi(x_n) - \varphi(x)}{d(x_n, x)} \le L - \lambda a, 
\end{align*}
which along with \eqref{ShrinkingSublevelSets} implies that
\begin{equation} \label{NegativeRSD_InfeasiblePoints}
  (\Phi_{\lambda})^{\downarrow}_A(x) < 0 \quad 
  \forall x \in S_{\lambda}(c) \setminus \Omega \quad \forall \lambda > \max\left\{ \frac{L}{a}, \lambda_0 \right\}.
\end{equation}
Fix $\lambda > \max\{ \lambda_0, L / a \}$. Let $x^* \in S_{\lambda}(c)$ be a point of local minimum of the penalised
problem \eqref{PenalizedProblem}. Then, as it is easy to check, $(\Phi_{\lambda})^{\downarrow}_A(x) \ge 0$, which with
the use of \eqref{NegativeRSD_InfeasiblePoints} implies that $x^* \in \Omega$. Hence taking into account the fact that 
$\Phi_{\lambda}(x) = \mathcal{I}(x)$ for any $x \in \Omega$ one obtains that $x^*$ is a point of local minimum of 
the problem $(\mathcal{P})$.

Let now $x^* \in S_{\lambda}(c)$ be a point of local minimum of the problem $(\mathcal{P})$. Then by applying
Lemma~\ref{Lemma_RetricSubregularity} and Remark~\ref{Remark_LipschitzEstimate} one gets that there exists $r > 0$ such
that for any $\lambda \ge L / a$ and $x \in B(x^*, r) \cap A$ one has
$$
  \Phi_{\lambda}(x) - \Phi_{\lambda}(x^*) = \mathcal{I}(x) - \mathcal{I}(x^*) + \lambda (\varphi(x) - \varphi(x^*))
  \ge - L \dist(x, \Omega) + \lambda a \dist(x, \Omega) \ge 0,
$$
i.e. $x^*$ is a point of local minimum of the penalised problem \eqref{PenalizedProblem}. \qed

\noindent\textbf{Proof of part 4.} Fix $\lambda > \max\{ L / a, \lambda_0 \}$. Let $x^* \in S_{\lambda}(c)$ be an
inf-stationary point of $\Phi_{\lambda}$ on $A$. Then by \eqref{NegativeRSD_InfeasiblePoints} one has $x^* \in \Omega$.
Hence taking into account the fact that $\Phi_{\lambda}(x) = \mathcal{I}(x)$ for any $x \in \Omega$ one can easily
check that $x^*$ is an inf-stationary point of $\mathcal{I}$ on $\Omega$.

Let now $x^* \in S_{\lambda}(c) \cap \Omega$ be an inf-stationary point of the function $\mathcal{I}$ on $\Omega$.
Note that one can suppose that $x^*$ is not an isolated point of the set $A$, since otherwise 
$(\Phi_{\lambda})^{\downarrow}_A(x^*) = + \infty$, i.e. $x^*$ is obviously an inf-stationary point of $\Phi_{\lambda}$
on $A$.

By the definition of the rate of steepest descent there exists a sequence $\{ x_n \} \subset A$ converging to $x^*$
such that
$$
  (\Phi_{\lambda})^{\downarrow}_A(x^*) 
  = \lim_{n \to \infty} \frac{\Phi_{\lambda}(x_n) - \Phi_{\lambda}(x^*)}{d(x_n, x^*)}.
$$
If there exists a subsequence $\{ x_{n_k} \} \subset \Omega$, then taking into account the fact that $\varphi(x) = 0$
for all $x \in \Omega$ one gets that
$$
  (\Phi_{\lambda})^{\downarrow}_A(x^*) 
  = \lim_{k \to \infty} \frac{\Phi_{\lambda}(x_{n_k}) - \Phi_{\lambda}(x^*)}{d(x_{n_k}, x^*)}
  = \lim_{k \to \infty} \frac{\mathcal{I}(x_{n_k}) - \mathcal{I}(x^*)}{d(x_{n_k}, x^*)} 
  \ge \mathcal{I}^{\downarrow}_{\Omega}(x^*) \ge 0.
$$
Thus, one can suppose that $\{ x_n \} \subset A \setminus \Omega$ and, moreover, $\Phi_{\lambda}(x_n) < c$ for all 
$n \in \mathbb{N}$, since otherwise there exists a subsequence $\{ x_{n_k} \}$ such that 
$$
  \Phi_{\lambda}(x_{n_k}) \ge c > \Phi_{\lambda}(x^*),
$$
which obviously implies that $(\Phi_{\lambda})^{\downarrow}_A(x^*) \ge 0$. Thus, 
$\{ x_n \} \subset S_{\lambda_0}(c) \setminus \Omega$.

Choose $L' \in (L, \lambda a)$. By applying Lemmas~\ref{Lemma_RetricSubregularity} and \ref{Lemma_LipschitzEstimate}
one obtains that
\begin{align*}
  \Phi_{\lambda}(x_n) - \Phi_{\lambda}(x^*) 
  &= \mathcal{I}(x_n) - \mathcal{I}(x^*) + \lambda \big( \varphi(x_n) - \varphi(x^*) \big) \\
  &\ge - L' \dist(x_n, \Omega) - (L' - L)d(x_n, x^*) + \lambda a \dist(x_n, \Omega)
  \ge - (L' - L)d(x_n, x^*)
\end{align*}
for any sufficiently large $n$. Dividing this inequality by $d(x_n, x^*)$, and passing to the limit as $n \to \infty$
one gets that $(\Phi_{\lambda})^{\downarrow}_A(x^*) \ge - (L' - L)$, which implies that
$(\Phi_{\lambda})^{\downarrow}_A(x^*) \ge 0$ due to the fact that $L' \in (L, \lambda a)$ was chosen arbitrarily. Thus,
$x^*$ is an inf-stationary point of $\Phi_{\lambda}$ on $A$, and the proof is complete. \qed

\section*{Appendix B. Integral Functionals and Nemytskii Operators}

In this section we obtain several results on the functional
$$
  \mathcal{I}(x, u) = \int_0^T \theta(x(t), u(t), t) \, dt + \zeta(x(T))
$$
and the nonlinear operator $F(x, u) = \dot{x}(\cdot) - f(x(\cdot), u(\cdot), \cdot)$ defined in
Section~\ref{Section_FreeEndpoint} that significantly simplify the verification of the assumptions of general theorems
on exact penalty functions in the case of optimal control problems. Namely, we obtain conditions under which 
the functional $\mathcal{I}(x, u)$ is correctly defined, finite valued, and Lipschitz continuous on bounded sets,
and the nonlinear operator $F(x, u) = \dot{x}(\cdot) - f(x(\cdot), u(\cdot), \cdot)$ defining the constraint 
$F(x, u) = 0$ maps $W_{1, p}^d(0, T) \times L_q^m(0, T)$ to $L_p^d(0, T)$ (note that the codomain of $F(x, u)$ is
$L_p^d(0, T)$, since $\dot{x}$ belongs to this space) and is differentiable in $x$. In order to obtain necessary and
sufficient conditions, we shall suppose that $x \in L_{\infty}^d[0, T]$ and $\zeta(\cdot) \equiv 0$. In the case when 
$x \in W^d_{1, p}(0, T)$ these conditions become only sufficient.

The following result is a particular case of \cite[Theorem~7.3]{FonsecaLeoni}.

\begin{proposition} \label{Prp_Functional_Correct}
Let the function $\theta$ be continuous and $\zeta(\cdot) \equiv 0$. Then $\mathcal{I}(x, u)$ is correctly defined and
finite valued for every $x \in L^d_{\infty}(0, T)$ and $u \in L_q^m(0, T)$ if and only if one of the two following
conditions is satisfied:
\begin{enumerate}
\item{$q = + \infty$;
}

\item{$1 \le q < + \infty$, and for every $R > 0$ there exist $C_R > 0$ and an a.e. nonnegative function 
$\omega_R \in L^1(0, T)$ such that
\begin{equation} \label{IntegrandGrowthCondition}
  |\theta(x, u, t)| \le C_R |u|^q + \omega_R(t)
\end{equation}
for a.e. $t \in [0, T]$ and for all $(x, u) \in \mathbb{R}^d \times \mathbb{R}^m$ with $|x| \le R$ (i.e. $\theta$
satisfies the growth condition of order $(q, 1)$).
}
\end{enumerate}
\end{proposition}

Next we obtain necessary and sufficient conditions for $\mathcal{I}(x, u)$ to be Lipschitz continuous on bounded sets,
which are needed for the verification of the assumptions of Theorem~\ref{THEOREM_COMPLETEEXACTNESS}. 

\begin{proposition} \label{Prp_LipschitzConditions}
Let $\zeta(\cdot) \equiv 0$, $\theta = \theta(x, u, t)$ be continuous, differentiable in $x$ and $u$, and let the
functions $\nabla_x \theta$ and $\nabla_u \theta$ be continuous as well. Suppose also that either $q = + \infty$ or
$\theta$ satisfies \eqref{IntegrandGrowthCondition}. Then the functional $\mathcal{I}(x, u)$ is Lipschitz continuous on
any bounded subset of $L^d_{\infty}(0, T) \times L^m_q(0, T)$ if and only if one of the following conditions is
satisfied:
\begin{enumerate}
\item{$q = +\infty$;
}

\item{$1 \le q < + \infty$, and for every $R > 0$ there exist $C_R > 0$, and a.e. nonnegative functions 
$\omega_R \in L^1(0, T)$ and $\eta_R \in L^{q'}(0, T)$ such that
\begin{equation}
  |\nabla_x \theta(x, u, t)| \le C_R |u|^q + \omega_R(t), \qquad
  |\nabla_u \theta(x, u, t)| \le C_R |u|^{q - 1} + \eta_R(t) \label{DerivativeGrowth}
\end{equation}
for a.e. $t \in [0, T]$ and for all $(x, u) \in \mathbb{R}^d \times \mathbb{R}^m$ with $|x| \le R$ (i.e. 
$\nabla_x \theta$ satisfies the growth condition of order $(q, 1)$, while $\nabla_u \theta$ satisfies the growth
condition of order $(q - 1, q')$).
}
\end{enumerate}
\end{proposition}

\begin{proof}
Let us prove the ``if'' part of the theorem first. For any $r > 0$ denote 
$$
  B_r = \Big\{ (x, u) \in L^d_{\infty}(0, T) \times L^m_q(0, T) \Bigm| \| x \|_{\infty} < r, \: \| u \|_q < r \Big\}.
$$
Choose $x, h \in L^d_{\infty}(0, T)$, $u, v \in L^m_q(0, T)$, and $\alpha > 0$. By the mean value theorem for a.e.
$t \in (0, T)$ there exists $\alpha(t) \in (0, \alpha)$ such that
\begin{multline*}
  \frac{1}{\alpha} \big( \theta(x(t) + \alpha h(t), u(t) + \alpha v(t), t) - \theta(x(t), u(t), t) \big) \\
  = \langle \nabla_x \theta(x(t) + \alpha(t) h(t), u(t) + \alpha(t) v(t), t), h(t) \rangle
  + \langle \nabla_u \theta(x(t) + \alpha(t) h(t), u(t) + \alpha(t) v(t), t), v(t) \rangle
\end{multline*}
for a.e. $t \in (0, T)$. From the continuity of $\nabla_x \theta$ and $\nabla_u \theta$ it follows that the right-hand
side of the above equality converges to 
$$
  \langle \nabla_x \theta(x(t), u(t), t), h(t) \rangle + \langle \nabla_u \theta(x(t), u(t), t), v(t) \rangle
$$
as $\alpha \to + 0$ for a.e. $t \in (0, T)$, and this function is measurable. Hence by applying one of the two
conditions of the proposition and Lebesgue's dominated convergence theorem one can easily verify that
\begin{equation} \label{IntegralFunctional_GateauxDeriv}
  \lim_{\alpha \to +0} \frac{\mathcal{I}(x + \alpha h, u + \alpha v) - \mathcal{I}(x, u)}{\alpha}
  = \int_0^T \Big( \langle \nabla_x \theta(x(t), u(t), t), h(t) \rangle
  + \langle \nabla_u \theta(x(t), u(t), t), v(t) \rangle \Big) \, dt,
\end{equation}
i.e. the functional $\mathcal{I}(x, u)$ is G\^{a}teaux differentiable, and its G\^ateaux derivative 
$\mathcal{I}'(x, u)[h, v]$ is equal to the expression above. Consequently, one has
\begin{equation} \label{DerivativeNorm}
  \| \mathcal{I}'(x, u) \| \le \big\| \nabla_x \theta(x(\cdot), u(\cdot), \cdot) \big\|_1 + 
  \big\| \nabla_u \theta(x(\cdot), u(\cdot), \cdot) \big\|_{q'}.
\end{equation}
Let $C \subset L_{\infty}^d(0, T) \times L_q^m(0, T)$ be a bounded set. Then $C \subseteq B_r$ for some $r > 0$. As is
well-known and easy to check, the functional $\mathcal{I}$ is Lipschitz continuous on $B_r$ with Lipschitz constant
$L > 0$ iff $L = \sup_{(x, u) \in B_r} \| \mathcal{I}'(x, u) \| < + \infty$. It remains to note that from
\eqref{DerivativeNorm} and the assumptions of the proposition it follows that this supremum is indeed finite.

Let us now prove the converse statement. We derive only the necessary growth conditions for the function 
$\nabla_u \theta$, since the derivation of the growth conditions for the function $\nabla_x \theta$ is essentially the
same (even slightly simpler) as in the case of $\nabla_u \theta$.

Let $1 < q < + \infty$, and fix some $x \in L^d_{\infty}(0, T)$ and $r > \| x \|_{\infty}$. Choose
$u \in L_q^m(0, T) \cap L_{\infty}^m(0, T)$ and $v \in (C[0, T])^m$ with $\| u \|_q < r$ and $\| v \|_q < r$,
where $C[0, T]$ is the space of of continuous functions defined on $[0, T]$. 

Choose $\alpha \in (0, 1]$. By the mean value theorem for a.e. $t \in (0, T)$ there exists $\alpha(t) \in (0, \alpha)$
such that
\begin{equation} \label{MenaValue_UDerivative}
  \frac{1}{\alpha} \big( \theta(x(t), u(t) + \alpha v(t), t) - \theta(x(t), u(t), t) \big)
  = \langle \nabla_u \theta(x(t), u(t) + \alpha(t) v(t), t), v(t) \rangle.
\end{equation}
The right-hand side of the above equality converges to $\nabla_u \theta(x(t), u(t), t)$ for a.e. $t \in (0, T)$ as
$\alpha \to +0$ due to the continuity of $\nabla_u \theta$. Hence integrating the left-hand side of
\eqref{MenaValue_UDerivative} from $0$ to $T$, and passing to the limit as $\alpha \to +0$ with the use of Lebesgue's
dominated convergence theorem and the fact that all functions $x$, $u$ and $v$ are essentially bounded one gets that
$$
  \lim_{\alpha \to + 0} \frac{\mathcal{I}(x, u + \alpha v) - \mathcal{I}(x, u)}{\alpha}
  = \int_0^T \langle \nabla_u \theta(x(t), u(t), t), v(t) \rangle \, dt.
$$
By our assumption the functional $\mathcal{I}(x, u)$ is Lipschitz continuous on $B_{2r}$ with Lipschitz constant 
$L_{2r} > 0$. Therefore
\begin{equation} \label{Derivative_vs_Lipschitz}
  \left| \int_0^T \langle \nabla_u \theta(x(t), u(t), t), v(t) \rangle \, dt \right| \le L_{2r} \| v \|_q.
\end{equation}
Choose now arbitrary $v \in L_q^m(0, T)$ with $\| v \| < r$. Since $C[0, T]$ is dense in $L^q(0, T)$ 
(see, e.g. \cite[Theorem~2.78]{FonsecaLeoni}), there exists a sequence $\{ v_n \} \subset (C[0, T])^m$
converging to $v$ in $L_q^m(0, T)$ and such that $\| v_n \| < r$ for all $n \in \mathbb{N}$. By applying inequality
\eqref{Derivative_vs_Lipschitz} with $v = v_n$ and passing to the limit as $n \to \infty$ one gets that
$$
  \left| \int_0^T \langle \nabla_u \theta(x(t), u(t), t), v(t) \rangle \, dt \right| \le L_{2r} \| v \|_q
  \quad \forall v \in L_q^m(0, T) \colon \| v \|_q < r
$$
(here we used the fact that the mapping $v \mapsto \int_0^T \langle \nabla_u \theta(x(t), u(t), t), v(t) \rangle \, dt$
is a continuous linear functional on $L_q^m(0, T)$ due to the essential boundedness of $x$ and $u$ and the continuity
of $\nabla_u \theta$). Taking the supremum over all $v \in L_q^m(0, T)$ with $\| v \|_q < r$ one gets that
\begin{equation} \label{DerivativeNorm_vs_LipschitzConstant}
  \int_0^T \big| \nabla_u \theta(x(t), u(t), t) \big|^{q'} \, dt \le L_{2r}^{q'}
\end{equation} 
for all $u \in L_q^m (0, T) \cap L_{\infty}^m (0, T)$ with $\| u \|_q < r$. Let us check that this inequality holds true
for any $u \in L_q^m (0, T)$ with $\| u \|_q < r$. Then taking into account the fact that $x$ and $r$ were chosen
arbitrarily one obtains that
$$
  \int_0^T \big| \nabla_u \theta(x(t), u(t), t) \big|^{q'} \, dt < + \infty
  \quad \forall (x, u) \in L_{\infty}^m(0, T) \times L_q^m(0, T).
$$
Consequently, by applying \cite[Theorem~7.3]{FonsecaLeoni} one gets that for any $R > 0$ there exist $C_R > 0$
and an a.e. nonnegative function $\omega_R \in L^1(0, T)$ such that
$| \nabla_u \theta(x, u, t) |^{q'} \le C_R |u|^q + \omega_R(t)$ for a.e. $t \in (0, T)$ and for all 
$(x, u) \in \mathbb{R}^d \times \mathbb{R}^m$ with $|x| < R$. Therefore for any such $x$, $u$ and $t$ one has
$$
  \big| \nabla_u \theta(x, u, t) \big| \le \big( C_R |u|^q + \omega_R(t) \big)^{1 / q'}
  \le C_R^{1 / q'} |u|^{q / q'} + \omega_R^{1/q'}(t),
$$
i.e. the desired growth condition (see~\eqref{DerivativeGrowth}) holds true (note that $q / q' = q - 1$).

Thus, it remains to prove that \eqref{DerivativeNorm_vs_LipschitzConstant} is valid for all $u \in L_q^m (0, T)$ with 
$\| u \|_q < r$. Fix any such $u$. For any $n \in \mathbb{N}$ define
$$
  u_n(t) = \begin{cases}
    u(t), & \text{if } |u(t)| \le n, \\
    n, & \text{otherwise.}
  \end{cases}
$$
Clearly, $u_n \in L_q^m(0, T) \cap L_{\infty}^m(0, T)$ and $\| u_n \|_q < r$ for any $n \in \mathbb{N}$. Furthermore, 
$|\nabla_u \theta(x(t), u_n(t), t)|^{q'}$ converges to $|\nabla_u \theta(x(t), u(t), t)|^{q'}$ as $n \to \infty$ for
a.e. $t \in (0, T)$. Hence by applying Fatou's Lemma (see, e.g.~\cite[Lemma~2.18]{Folland}) and
\eqref{DerivativeNorm_vs_LipschitzConstant} one obtains that
$$
  \int_0^T |\nabla_u \theta(x(t), u(t), t)|^{q'} \, dt \le
  \liminf_{n \to \infty} \int_0^T |\nabla_u \theta(x(t), u_n(t), t)|^{q'} \, dt \le L_{2r}^{q'},
$$
and the proof of the case $1 < q < + \infty$ is complete.

To obtain the necessary growth condition in the case $q = 1$ note that in this case inequality
\eqref{Derivative_vs_Lipschitz} holds true for any $v \in L_1^m(0, T)$ with $\| v \|_1 < r$. Taking 
the supremum over all such $v$ one gets that $|\nabla_u \theta(x(t), u(t), t)| \le L_{2r}$ for a.e. $t \in (0, T)$. Let 
$u \in L_1^m (0, T)$ with $\| u \|_1 < r$ be arbitrary. By definition $u_n \in L_{\infty}^m(0, T)$ for any 
$n \in \mathbb{N}$. Therefore $|\nabla_u \theta(x(t), u_n(t), t)| \le L_{2r}$ for a.e. $t \in (0, T)$ and for all 
$n \in \mathbb{N}$, which obviously implies that $|\nabla_u \theta(x(t), u(t), t)| \le L_{2r}$ for a.e. $t \in (0, T)$.
Hence taking into account the fact that $x$ and $r$ were chosen arbitrarily one obtains that 
$\nabla_u \theta(x(\cdot), u(\cdot), \cdot) \in L_{\infty}^m(0, T)$ for any $x \in L_{\infty}^d(0, T)$ and 
$u \in L_1^m(0, T)$. Utilising this result one can easily obtain the required growth condition on $\nabla_u \theta$ in
the case $q = 1$.
\end{proof}

\begin{remark}
Note that in the case $q = 1$ the second inequality in \eqref{DerivativeGrowth} simply means that for any $R > 0$ there
exists $C_R > 0$ such that $|\nabla_u \theta(x, u, t)| \le C_R$ for a.e. $t \in (0, T)$ and for all 
$(x, u) \in \mathbb{R}^d \times \mathbb{R}^m$ with $|x| \le R$.
\end{remark}

Let us also point conditions under which the nonlinear operator $F(x, u) = \dot{x}(\cdot) - f(x(\cdot), u(\cdot),
\cdot)$ is correctly defined and differentiable in $x$. The following result is a simple generalisation of the standard
theorem on Nemytskii operators (see, e.g. \cite{AppellZabrejko}).

\begin{proposition} \label{Prp_DiffEqConstr_Correct}
Let $f$ be continuous. Then the Nemytskii operator $(x(\cdot), u(\cdot)) \mapsto f(x(\cdot), u(\cdot), \cdot)$ maps
$L_{\infty}^d (0, T) \times L_q^m (0, T)$ to $L_p^d(0, T)$ if and only if one of the two following conditions is
satisfied:
\begin{enumerate}
\item{$q = + \infty$;
}

\item{$1 \le q < + \infty$, and for every $R > 0$ there exist $C_R > 0$ and an a.e. nonnegative function 
$\omega_R \in L^p(0, T)$ such that
\begin{equation} \label{DiffEqGrowthCond}
  |f(x, u, t)| \le C_R |u|^\frac{q}{p} + \omega_R(t)
\end{equation}
for a.e. $t \in [0, T]$ and for all $(x, u) \in \mathbb{R}^d \times \mathbb{R}^m$ with $|x| \le R$ (i.e. $f$ satisfies
the growth conditions of order $(q/p, p)$).
}
\end{enumerate}
Moreover, if one of these conditions is satisfied, then $F$ is a continuous nonlinear operator from
$W_{1, p}^d(0, T) \times L_q^m(0, T)$ to $L_p^d(0, T)$.
\end{proposition}

\begin{proof}
Let $1 \le q < + \infty$ (the validity of the statement in the case $q = + \infty$ follows directly from the continuity
of the function $f$). By definition the operator $(x(\cdot), u(\cdot)) \mapsto f(x(\cdot), u(\cdot), \cdot)$ 
maps $L_{\infty}^d (0, T) \times L_q^m(0, T)$ to $L_p^d(0, T)$ iff
$$
  \int_0^T |f(x(t), u(t), t)|^p \, dt < + \infty 
  \quad \forall (x, u) \in L_{\infty}^d (0, T) \times L_q^m (0, T).
$$
Hence by applying \cite[Theorem~7.3]{FonsecaLeoni} one obtains that this operator maps 
$L_{\infty}^d (0, T) \times L_q^m(0, T)$ to $L_p^d(0, T)$ iff for every $R > 0$ there exist $C_R > 0$ and 
an a.e. nonnegative function $\omega_R \in L^1(0, T)$ such that
$$
  |f(x, u, t)|^p \le C_R |u|^q + \omega_R(t)
$$
for a.e. $t \in [0, T]$ and for all $(x, u) \in \mathbb{R}^d \times \mathbb{R}^m$ with $|x| \le R$. If this inequality
is satisfied, then
$$
  |f(x, u, t)| \le \big( C_R |u|^q + \omega_R(t) \big)^{1 / p} 
  \le C_R^{1 / p} |u|^{\frac{q}{p}} + \omega_R^{\frac{1}{p}}(t),
$$
which implies the validity of the ``only if'' part of the proposition (note that $\omega_R^{1/p} \in L^p(0, T)$).
Conversely, if 
$$
  |f(x, u, t)| \le C_R |u|^{\frac{q}{p}} + \omega_R(t)
$$
for some $C_R > 0$ and $\omega_R \in L^p(0, T)$, then
$$
  |f(x, u, t)|^p \le \big( C_R |u|^{\frac{q}{p}} + \omega_R(t) \big)^p 
  \le 2^p \big( C_R^p |u|^q + \omega_R^p(t) \big)
$$
which implies the validity of the ``if'' part of the proposition.

Let one of the conditions be satisfied. From the fact that every $x \in W^d_{1, p}(0, T)$ belongs to 
$L_{\infty}^d (0, T)$ it follows that the operator $(x(\cdot), u(\cdot)) \mapsto f(x(\cdot), u(\cdot), \cdot)$ maps 
$W_{1, p}^d(0, T) \times L_q^m(0, T)$ to $L_p^d(0, T)$, and, therefore, so does the operator $F$. The continuity of this
operator can be easily verified in the case $1 \le q < + \infty$ with the use of Vitali's theorem characterising
convergence in $L^p$ spaces (cf. the proof of \cite[Theorem~5.1]{Precup}), and it can be proved via a simple
$\varepsilon$-$\delta$ argument in the case $q = + \infty$.
\end{proof}

\begin{proposition} \label{Prp_DiffEqConstr_Diferentiable}
Let $f = f(x, u, t)$ be continuous, differentiable in $x$, and let the function $\nabla_x f$ be continuous. Suppose also
that either $q = + \infty$ or inequality \eqref{DiffEqGrowthCond} holds true. Then 
the Nemytskii operator $G_u \colon L_{\infty}^d(0, T) \to L_p^d(0, T)$ defined as 
$G_u(x) = f(x(\cdot), u(\cdot), \cdot)$ is G\^{a}teaux differentiable at every point $x \in L_{\infty}^d(0, T)$, and its
G\^{a}teaux derivative has the form
\begin{equation} \label{NemytskiiOperator_Derivative}
  G_u'(x)[h] = \nabla_x f(x(\cdot), u(\cdot), \cdot) h(\cdot) \quad \forall h \in  L_{\infty}^d(0, T)
\end{equation}
for any $u \in L_q^m (0, T)$ if and only if one of the following conditions is satisfied
\begin{enumerate}
\item{$q = +\infty$;
}

\item{$1 \le q < + \infty$, and for any $R > 0$ there exist $C_R > 0$ and an a.e. nonnegative function
$\omega_R \in L^p(0, T)$ such that
\begin{equation} \label{DiffEqDerivatives_GrowthCond}
  |\nabla_x f(x, u, t)| \le C_R |u|^\frac{q}{p} + \omega_R(t), \quad
\end{equation}
for a.e. $t \in (0, T)$ and for all $(x, u) \in \mathbb{R}^d \times \mathbb{R}^m$ with $|x| \le R$ (i.e. $\nabla_x f$
satisfies the growth condition of order $(q / p, p)$).
}
\end{enumerate}
\end{proposition}

\begin{proof}
If one the conditions is satisfied, then by applying Lebesgue's dominated convergence theorem one can easily check that
$G_u(x)$ is G\^{a}teaux differentiable, and \eqref{NemytskiiOperator_Derivative} holds true. Conversely, if $G_u(x)$
is G\^{a}teaux differentiable, and \eqref{NemytskiiOperator_Derivative} holds true, then, as it is easily seen, 
the operator $(x, u) \mapsto \nabla_x f(x(\cdot), u(\cdot), \cdot)$ maps $L_{\infty}^d(0, T) \times L_q^m(0, T)$ to
$L_p^{d \times d} (0, T)$. Hence arguing in the same way as in the proof of Proposition~\ref{Prp_DiffEqConstr_Correct}
one obtains that inequality \eqref{DiffEqDerivatives_GrowthCond} is valid.
\end{proof}

\begin{remark}
Arguing in a similar way to the proof of Proposition~\ref{Prp_DiffEqConstr_Diferentiable} one can check that the
functional $\mathcal{I}(x, u)$ is G\^{a}teaux differentiable at every point 
$(x, u) \in L_{\infty}^d(0, T) \times L_q^m(0, T)$ and its G\^{a}teaux derivative has the natural form
\eqref{IntegralFunctional_GateauxDeriv} iff one of the two conditions of Proposition~\ref{Prp_LipschitzConditions} are
satisfied. Therefore, if $\theta(x, u, t)$ is continuous, differentiable in $x$ and $u$, and the functions 
$\nabla_x \theta$ and $\nabla_u \theta$ are continuous, then the functional $\mathcal{I}(x, u)$ is Lipschitz continuous
on bounded subsets of the space $L_{\infty}^d(0, T) \times L_q^m(0, T)$ iff it is G\^{a}teaux differentiable at every
point of this space, and its G\^{a}teaux derivative has the natural form \eqref{IntegralFunctional_GateauxDeriv}.
\end{remark}

Let us also prove a simple auxiliary result on the resolvent of a Volterra-type integral equation. This result is
well-known. Nevertheless, we briefly outline its proof for the sake of completeness, and because of the fact that 
the equation that we analyse slightly differs from the classical one \cite{Gripenberg} (instead of integrating from $0$
to $t$ we integrate from $t$ to $T$). 

Denote by $I$ the identity operator, and define $(\mathcal{K}_y x)(t) = \int_t^T y(t, s) x(s) \, ds$ for a.e. 
$t \in (0, T)$, where $x \colon (0, T) \to \mathbb{R}^d$ and $y \colon (0, T)^2 \to \mathbb{R}^{d \times d}$ are
measurable functions such that $y(t, \cdot) x(\cdot) \in L_1^d(0, T)$ for a.e. $t \in (0, T)$.

\begin{lemma} \label{Lemma_BoundedResolvent}
Let a function $y \colon (0, T)^2 \to \mathbb{R}^{d \times d}$ be measurable and satisfy the inequality 
$|y(t, s)| \le y_0(s)$ for a.e. $t, s \in (0, T)$ and for some a.e. nonnegative function $y_0 \in L^{p'}(0, T)$. Then
the operator $I - \mathcal{K}_y$ maps $L_p^d (0, T)$ to $L_p^d (0, T)$ and is invertible. Furthermore, there exists a
continuous function $\omega \colon [0, + \infty) \times [0, + \infty) \to [0, +\infty)$ such that 
$\| (I - \mathcal{K}_y)^{-1} \| \le \omega(T, \| y_0 \|_{p'})$.
\end{lemma}

\begin{proof}
Fix a measurable function $y(t, s)$ satisfying the assumptions of the lemma for some $y_0 \in L^{p'}(0, T)$. By applying
H\"{o}lder's inequality one can easily check that the operator $\mathcal{K}_y$ maps $L_p^d (0, T)$ to $L_p^d (0, T)$,
and $|(\mathcal{K}_y x)(t)| \le \| y_0 \|_{p'} \| x \|_p$ for a.e. $t \in (0, T)$. 

It is well-known and easy to check that if the Neumann series $\sum_{n = 0}^{\infty} \mathcal{K}_y^n$ converges in the
operator norm, then its limit is the inverse of $I - \mathcal{K}_y$. Let us check the convergence of this series.
Indeed, with the use of H\"{o}lder's inequality one gets that
$$
  \big| (\mathcal{K}_y^2 x) (t) \big| = \bigg| \int_t^T y(t, s) (\mathcal{K}_y x)(s) \, ds \bigg|
  \le \| y_0 \|_{p'} \| x \|_p \int_t^T |y_0(s)| \, ds 
  \le \| y_0 \|_{p'}^2 \| x \|_p |T - t|^{\frac{1}{p}}
$$
for a.e. $t \in (0, T)$. Similarly, one has
$$
  \big| (\mathcal{K}_y^3 x) (t) \big| = \bigg| \int_t^T y(t, s) (\mathcal{K}_y^2 x)(s) \, ds \bigg|
  \le \| y_0 \|_{p'}^2 \| x \|_p \int_t^T |y_0(s)| |T - s|^{\frac{1}{p}} \, ds 
  \le \| y_0 \|_{p'}^3 \| x \|_p \left( \frac{1}{2} |T - t|^2 \right)^{\frac{1}{p}}
$$
for a.e. $t \in (0, T)$. By induction one can easily check that
$$
  \big| (\mathcal{K}_y^n x) (t) \big| 
  \le \| y_0 \|_{p'}^n \| x \|_p \left( \frac{1}{(n - 1)!} |T - t|^{n - 1} \right)^{\frac{1}{p}}
$$
for a.e. $t \in (0, T)$ and for any $n \in \mathbb{N}$. Therefore
$$
  \big\| \mathcal{K}_y^n x \big\|_p \le \left( \frac{1}{n!} T^n \right)^{\frac{1}{p}} \| y_0 \|_{p'}^n \| x \|_p, \quad
  \| \mathcal{K}_y^n \| \le \left( \frac{1}{n!} T^n \right)^{\frac{1}{p}} \| y_0 \|_{p'}^n
$$
for all $n \in \mathbb{N}$. Consequently, the Neumann series $\sum_{n = 0}^{\infty} \mathcal{K}^n_y$ converges, and the
norm of its limit does not exceed $\omega(T, \| y_0 \|_{p'})$, where
$$
  \omega(\tau, \xi) = \sum_{n = 0}^{\infty} \left( \frac{1}{n!} \tau^{n} \right)^{\frac{1}{p}} \xi^n.
$$
It remains to note that the series in the definition of $\omega$ converges to a continuous function uniformly on bounded
sets by the Weierstrass M-test.
\end{proof}

\end{document}